\documentclass[twoside,english,final]{amsart}
\usepackage{amsmath}
\usepackage{amsthm}
\usepackage{amssymb}
\usepackage{bm}

\makeatletter

\theoremstyle{plain}
\newtheorem{thm}{\protect\theoremname}[section]
\theoremstyle{definition}
\newtheorem{defn}[thm]{\protect\definitionname}

\theoremstyle{plain}
\newtheorem{lem}[thm]{\protect\lemmaname}
\theoremstyle{plain}
\newtheorem{prop}[thm]{\protect\propositionname}
\theoremstyle{plain}
\newtheorem{cor}[thm]{\protect\corollaryname}

\theoremstyle{remark}

\newtheorem*{acknowledgement*}{Acknowledgement}

\usepackage{tikz-cd}

	 \renewcommand{\theenumi}{\emph{\roman{enumi})}} 
	\numberwithin{equation}{section}

\DeclareMathOperator{\Irr}{Irr}

\DeclareMathOperator{\Ind}{Ind}
\DeclareMathOperator{\Res}{Res}
\DeclareMathOperator{\res}{res}

\DeclareMathOperator{\Stab}{Stab}
\DeclareMathOperator{\Tr}{tr}

\DeclareMathOperator{\Ker}{Ker}
\DeclareMathOperator{\Hom}{Hom}

\DeclareMathOperator{\GL}{GL}

\DeclareMathOperator{\SL}{SL}
\DeclareMathOperator{\M}{M}

\newcommand\leftexp[2]{\hspace{1pt}{\vphantom{#2}}^{#1}\hspace{-1pt}{#2}}

\newcommand{\C}{\ensuremath{\mathbb{C}}}

\newcommand{\N}{\ensuremath{\mathbb{N}}}
\newcommand{\Q}{\ensuremath{\mathbb{Q}}}

\newcommand{\Z}{\ensuremath{\mathbb{Z}}}

\newcommand{\cB}{\ensuremath{\mathcal{B}}}
\newcommand{\cC}{\ensuremath{\mathcal{C}}}
\newcommand{\cD}{\ensuremath{\mathcal{D}}}

\newcommand{\cG}{\ensuremath{\mathcal{G}}}
\newcommand{\cH}{\ensuremath{\mathcal{H}}}

\newcommand{\cO}{\ensuremath{\mathcal{O}}}

\newcommand{\cT}{\ensuremath{\mathcal{T}}}
\newcommand{\cS}{\ensuremath{\mathcal{S}}}

\renewcommand{\bar}[1]{\mkern 1.5mu\overline{\mkern-1mu#1\mkern-1mu}\mkern 1.5mu}

\newcommand{\llbracket}{[\mkern-2mu [}
\newcommand{\rrbracket}{]\mkern-2mu ]}

\date{\today}

		\newcommand{\tpartial}[2]{\widetilde{Z}_{#1}^{#2}(s)}
		\newcommand{\lan}{\mathcal{L}}
		\newcommand{\struc}{\mathcal{M}}
		\newcommand{\Lan}{\lan_{\mathrm{an}}}
		
		\newcommand{\man}{\struc_{\mathrm{an}}}
		\newcommand{\eqrel}{\mathbin{\mathcal{E}}}
		\newcommand{\tuple}[1]{{\bm{\mathrm{#1}}}}

		\newcommand{\coho}[1]{\mathrm{H}^{#1}}
		\newcommand{\cocy}[1]{\mathrm{Z}^{#1}}
		\newcommand{\cobo}[1]{\mathrm{B}^{#1}}
		\newcommand{\onecocyp}{\widetilde{Z}_p}
		\newcommand{\onecobop}{\widetilde{B}_p}

		\newcommand{\stgroup}{K}
		\newcommand{\uniform}{N}
		
		\newcommand{\Kp}{{K_p}}
		\newcommand{\Lp}{{L_p}}
		
		\newcommand{\Gammap}{\Gamma_p}
		
		\newcommand{\indexPN}{r}
		\newcommand{\indexKpN}{{\indexPN}}
		\newcommand{\indexLpN}{{\indexPN'}}
		
		\newcommand{\indexLN}{{u'}}
		\newcommand{\indexGN}{{m}}

\newcommand{\twist}{\mathbin{\sim}} 
\newcommand{\Gtwist}{\mathrel{\sim_G}} 
	
\newcommand{\twirr}{\widetilde{\Irr}}
\newcommand{\twirrparams}{\twirr\vphantom{\Irr}^{c,c'}_{L,K,\Gamma}(N)}

\newcommand{\tic}[1]{\widetilde{#1}} 

\DeclareMathOperator{\Func}{Func} 
\DeclareMathOperator{\Lin}{Lin} 
\DeclareMathOperator{\twind}{\widetilde{\Ind}}  
\DeclareMathOperator{\PIrr}{PIrr} 

\newcommand{\ptwirr}{\widetilde{\PIrr}}

\usepackage{babel}
\providecommand{\corollaryname}{Corollary}
\providecommand{\definitionname}{Definition}
\providecommand{\lemmaname}{Lemma}
\providecommand{\propositionname}{Proposition}
\providecommand{\theoremname}{Theorem}

\usepackage[textsize=footnotesize, obeyFinal]{todonotes}

\newcommand{\thmMain}{Theorem~1.1}				
\newcommand{\corMain}{Corollary~1.2}				

\newcommand{\defMan}{Section~2.2}				
\newcommand{\lemInterpret}{Lemma~2.4}			
\newcommand{\thmRationalSeries}{Theorem~2.7}		

\newcommand{\defStrongExt}{Definition~3.1}			
\newcommand{\thmCliffordMap}{Theorem~3.3}			
\newcommand{\lemSameFs}{Lemma~3.4}			
\newcommand{\lemCliffordExt}{Lemma~3.6}			
\newcommand{\lemBasicCoho}{Lemma~3.8}			

\newcommand{\lemJaikinsProp}{Lemma~4.1}			

\newcommand{\secRedDegOne}{Section~5}			
\newcommand{\propRedLinear}{Proposition~5.2}		
\newcommand{\corSurjCoho}{Corollary~5.3}			

\newcommand{\subsecGoodBases}{Section~6.1}		
\newcommand{\defGoodBasis}{Definition~6.1}			
\newcommand{\propDc}{Proposition~6.9}				
\newcommand{\prerequisites}{Propositions 5.2 and 6.9}	
\begin{document}

\title[Rationality of twist zeta functions]{Rationality of twist representation zeta functions of compact $p$-adic analytic groups}

\author{Alexander Stasinski and Michele Zordan}
\address{Department of Mathematical Sciences, Durham University,
	Durham, DH1 3LE, UK}
\email{alexander.stasinski@durham.ac.uk}

\address{Department of Mathematics, Imperial College London, London SW7 2AZ, UK.}
\email{mzordan@imperial.ac.uk}

\begin{abstract}
We prove that for any twist rigid compact $p$-adic analytic group $G$, its
twist representation zeta function is a finite sum of terms $n_{i}^{-s}f_{i}(p^{-s})$,
where $n_{i}$ are natural numbers and $f_{i}(t)\in\Q(t)$ are
rational functions. Meromorphic continuation and rationality of the abscissa of the zeta function follow as corollaries. If $G$ is moreover a pro-$p$ group, we prove
that its twist representation zeta function is rational in $p^{-s}$. 
To establish these results we develop a Clifford theory for twist isoclasses of representations, including a new cohomological invariant of a twist isoclass.
\end{abstract}

\maketitle


\section{Introduction}

The representation zeta series of a group $G$ is the formal Dirichlet series
\[
Z_G(s)=\sum_{n=1}^{\infty}r_n(G)n^{-s}=\sum_{\rho\in \Irr(G)}\rho(1)^{-s},
\]
where $r_n(G)$ is the number of isomorphism classes of
(continuous, if $G$ is topological) complex irreducible $n$-dimensional
representations of $G$ (assumed to be finite for each $n$) and  $\Irr(G)$ denotes the set of irreducible characters of $G$. If the sequence $R_N(G)=\sum_{i=1}^{N}r_i(G)$ grows at most polynomially, $Z_G(s)$ defines a holomorphic function $\zeta_{G}(s)$ on some right half-plane of $\C$, which is called the representation zeta function of~$G$.
Representation zeta functions have been studied in a number of papers; see the references in \cite{Rationality1}.

A group $G$ is called representation rigid if the number $r_n(G)$ is finite, for each $n$. This holds for large families of groups, such as semisimple Lie groups, arithmetic groups and semisimple $p$-adic groups (see, for instance, \cite{Lubotzky-Martin} and \cite{Larsen-Lubotzky}). 
There are significant classes of groups that are not representation rigid, for example torsion-free nilpotent groups or reductive $p$-adic groups with infinitely many $1$-dimensional representations, like $\GL_n(\Z_p)$. These groups therefore do not possess a representation zeta function in the usual sense. Nevertheless, it turns out that in many cases the number $\widetilde{r}_n(G)$ of irreducible representations (up to isomorphism) of dimension $n$ up to one-dimensional twists (i.e., tensoring by one-dimensional representations) is finite for all $n$. We call such groups \emph{twist rigid}. Examples of twist rigid groups include finitely generated nilpotent groups and most (possibly all) reductive compact $p$-adic groups. The latter is work in progress of the authors joint with B.~Martin.
 For a twist rigid group $G$, one can define the Dirichlet series 
\[
\widetilde{Z}_G(s)=\sum_{n=1}^{\infty}\widetilde{r}_n(G)n^{-s}
\]
and its meromorphic continuation $\widetilde{\zeta }_G(s)$ (where it exists), called the \emph{twist (representation) zeta series/function}, respectively.

Following the terminology in \cite{Rationality1}, we say that a Dirichlet series with integer coefficients is \emph{virtually rational in
$p^{-s}$} if, as an element of $\Z\llbracket p_1^{-s},p_2^{-s},\dots\rrbracket$, where $p_1,p_2,\dots$ are the primes in $\N$, it is of the form
\begin{equation}
\sum_{i=1}^{k}n_{i}^{-s}f_{i}(p^{-s}),
\end{equation}
for some natural numbers $k$ and $n_{i}$ and rational functions $f_{i}(t)\in\Q(t)$.
If $Z_G(s)$ defines a zeta function $\zeta_{G}(s)$, we say that $\zeta_{G}(s)$ is virtually rational in $p^{-s}$ if $Z_G(s)$ is. When speaking informally, we will often say that a zeta series/function is (virtually) rational, that is, omitting the specification ``in $p^{-s}$''.

In \cite{Stasinski-Voll-Tgrps} the first author and Voll proved rationality of the local factors of twist zeta functions of torsion-free finitely generated nilpotent groups associated with certain group schemes 
when a suitable Kirillov orbit method can be applied and in
 \cite{hrumar2015definable} Hrushovski, Martin and Rideau 
proved (among other things) rationality of local factors of twist representation zeta functions for all finitely generated nilpotent groups. Examples of these zeta functions have been computed in \cite{Stasinski-Voll-Tgrps,zor2022univariate}, and their abscissae of convergence and analytic properties  have been investigated in \cite{dunvol2017uniform}.

The study of twist representation zeta functions of compact $p$-adic groups was initiated by the first author and H\"as\"a in \cite{Hasa-Stasinski}, who proved in particular that 
$\GL_n(\cO)$ is twist rigid (where $\cO$ is any compact discrete valuation ring) and explicitly computed the twist zeta function of $\GL_2(\cO)$ when $2$ is a unit in $\cO$. 
When the characteristic of $\cO$ does not divide 
$n$, Stasinski and H\"as\"a also proved that the abscissa of 
convergence of $\widetilde{\zeta }_{\GL_n(\cO)}$ coincides with the abscissa of convergence of $\zeta_{\SL_n(\cO)}$ (\cite[Proposition~3.4]{Hasa-Stasinski}). 
Computing the latter abscissae is an active area of research:
see, for example, the work by Larsen and Lubotzky \cite{Larsen-Lubotzky}, by Aizenbud and Avni \cite{Aizenbud-Avni-2016}, by Budur \cite{bud2021rational} for lower and upper bounds 
and the work by Avni, Klopsch, Onn, and Voll \cite{AKOV-Duke, AKOV-GAFA,AKOV-ProcLMS} and by Zordan \cite{zor2016adjoint} for computations in the cases $n = 2,3$.

Unlike the pro-$p$ completions of finitely generated nilpotent groups, 
compact $p$-adic groups do not have rational twist zeta function, in general. However, in the present paper we prove:

\begin{thm}
	\label{thm:Main-twist}Let $G$ be a twist rigid compact $p$-adic analytic group.
	Then $\widetilde{\zeta}_{G}(s)$ is virtually rational in $p^{-s}$. If in addition
	$G$ is pro-$p$, then $\widetilde{\zeta}_{G}(s)$ is rational in $p^{-s}$.
\end{thm}
By the same argument as for \cite[\corMain]{Rationality1}, this theorem has the following consequences.
\begin{cor}\label{cor:Main-twist}
	Let $G$ be a twist rigid compact $p$-adic analytic group. Then the following holds regarding $\widetilde{\zeta}_{G}(s)$:
	\begin{enumerate}
		\item it extends meromorphically to the whole complex plane,
		\item it has an abscissa of convergence which is a rational number.
	\end{enumerate}
\end{cor}


In \cite{Rationality1}, we proved the analogous results (cf.\ \cite[\thmMain\ and \corMain]{Rationality1}) 
for the representation zeta functions of representation rigid groups (a different proof, for $p\neq 2$, had been given earlier by Jaikin-Zapirain \cite{Jaikin-zeta}). 
Although the proof of Theorem~\ref{thm:Main-twist} relies in part on results that were originally proved in \cite{Rationality1} (namely \cite[\prerequisites]{Rationality1}), the present case of twist zeta functions is substantially more difficult and requires several new ideas. 

Section~\ref{sec:Twist-iso-Clifford} defines and studies restriction and induction of what we call $G$-twist classes of characters in the presence of a normal subgroup. Here we let $G$ be an arbitrary profinite group and $N$ a normal subgroup of finite index. For any subgroup $H$ of $G$, we say that $\lambda,\delta\in\Irr(H)$
are \emph{$G$-twist equivalent}	if $\lambda=\delta\psi|_{H}$, for some character $\psi$ of $G$ of degree one (see Definition~\ref{def:G-twist}). 

Let now $H$ and $H'$ be subgroups of $G$ such that $H\leq H'$ and such that $H$ contains the stabiliser in $G$ of some $\theta \in \Irr(N)$. Then the Clifford correspondence says that induction gives a bijection between irreducible characters of $H$ lying over $\theta$ and irreducible characters of $H'$ lying over $\theta$. As we will see, this immediately implies that induction of $G$-twist classes is a surjective map. However, in contrast to the classical Clifford correspondence, induction of $G$-twist classes is not necessarily injective. It is for this reason that our proof of Theorem~\ref{thm:Main-twist} requires new methods in addition to those used in the proof of \cite[\thmMain]{Rationality1}.

The main new ingredient needed is an invariant $\cT_{L,K,\Gamma}(\tic{\theta})$ attached to the $G$-twist class $\tic{\theta}$ of a $\theta \in \Irr(N)$ that controls precisely when two $G$-twist classes induce to the same $G$-twist class. This invariant is an element in the group cohomology
\[
	\coho{1}(L/N,F_K/\Gamma),
\]
where $L$ is the stabiliser of $\tic{\theta}$ in $G$, $K$ is the stabiliser of $\theta$ in $G$ (so that $K\trianglelefteq L$), $F_K$ is the set of functions $K/N\rightarrow \C^{\times}$, 
$\Gamma$ is a certain subgroup of $\Hom(K/N,\C^{\times})$ and the action of $L/N$ on $F_K/\Gamma$ is the co-adjoint action (see Section~\ref{subsec:The-function-bar-mu} for the definitions).
We give a quick idea of how $\cT_{L,K,\Gamma}(\tic{\theta})$ is defined.  By definition of $L$, any $g\in L$ fixes $\theta$ up to $G$-twist, that is, for any $g\in L$ there is a character $\psi_g$ of $G$ of degree one such that $\leftexp{g}{\theta}=\theta\psi_{g}|_{N}$. Now let $\hat{\theta}$ be a projective character (i.e., the character of a projective representation) 
of $K$ strongly extending $\theta$ (see \cite[\defStrongExt]{Rationality1}). Then both $\leftexp{g}{\hat{\theta}}$ and $\hat{\theta}\psi_{g}|_{K}$ strongly extend $\leftexp{g}{\theta}$, so there exists a function $\mu(g):K/N \rightarrow \C^{\times}$ such that
\[
\leftexp{g}{\hat{\theta}}=\hat{\theta}\psi_{g}|_{K}\cdot\mu(g).
\]
The goal of Section~\ref{sec:Twist-iso-Clifford} is then to prove that the function $g \mapsto \mu(g)$ gives rise to a unique element in $\coho{1}(L/N,F_K/\Gamma)$, where the ambiguity in the choice of strong extension $\hat{\theta}$ has been accounted for by quotienting out by $1$-coboundaries, and the the ambiguity in the choice of $\psi_g$ has been accounted for by quotienting out by $\Gamma$. At the same time, it is shown that the resulting cohomology class only depends on the class $\tic{\theta}$, and not on the choice of representative $\theta$.

The next step, carried out in Section~\ref{sec:Reduction-to-pro-p_twisted_case}, is to show that the invariant $\cT_{L,K,\Gamma}(\tic{\theta})$ is determined by
$\cC_{K_p}(\tic{\theta})$ together with
$\cT_{L_p,K_p,\Gamma_p}(\tic{\theta})$. Here $\cC_{K_p}$ is a function with values in $\coho{2}(K_p/N,\C^{\times})$ defined in \cite{Rationality1} but here considered as a function on $G$-twist classes, $L_p$ and $K_p$ are pro-$p$ Sylow subgroups of $L$ and $K$, respectively and $\Gamma_p$ is the image of $\Gamma$ under restriction of homomorphisms to $K_p/N$. Here it is assumed that $N$ is a normal pro-$p$ 
subgroup of $G$ but eventually $N$ will be a uniform subgroup. The reasons for reducing to pro-$p$ Sylow subgroups is the same as for the reduction of $\coho{2}(K/N,\C^{\times})$ to $\coho{2}(K_p/N,\C^{\times})$ 
 in the proof of \cite[\thmMain]{Rationality1}, but in the  twist zeta setting considered in the present paper, the reduction is more complicated and uses very different arguments.

In Section~\ref{sec:Reduction partial twist} we use the main result of the previous section (Proposition~\ref{prop:red_coeff_to_Sylow}) to prove that Theorem~\ref{thm:Main-twist} follows from the rationality of the partial twist zeta series $\tpartial{N;L, K, \Gamma}{c, c'}$. Finally, Section~\ref{sec:rationality_partial_tw} proves rationality of the partial twist zeta series by, among other things, showing that the condition $\cT_{L_p,K_p,\Gamma_p}(\tic{\theta})=c'$, for $c'\in \coho{1}(L/N,F_K/\Gamma)$, can be expressed as a definable condition on a suitable definable set.

\section{Twist classes and Clifford theory\label{sec:Twist-iso-Clifford}}

From now on and throughout the rest of this paper, we will develop
results that will lead up to the proof of Theorem~\ref{thm:Main-twist}. 
The main goal of the present section is to define a cohomology class $\cT_{L,K,\Gamma}(\tic{\theta})$ 
attached to a twist class $\tic{\theta}$ of $N$. In the following section, 
we will show that $\cT_{L,K,\Gamma}(\tic{\theta})$ controls the number of 
$G$-twist classes of $L$ lying above a given $\tic{\theta}$. In this sense, 
the function $\cT_{L,K,\Gamma}$ can be thought of as an analogue of the 
function $\cC_K$ introduced  in our previous joint work:

\begin{thm}[{\cite[Theorem~3.3]{Rationality1}}]
\label{thm:Clifford-map}Let $\Theta$ be an irreducible representation
of $N$ fixed by $K\leq G$. There exists a projective representation $\Pi$
of $K$ which strongly extends $\Theta$. Let $\hat{\alpha}$ be the
factor set of $\Pi$. Then $\hat{\alpha}$ is constant on cosets in
$K/N$, so we have a well-defined element $\alpha\in \cocy{2}(K/N)$
given by
\[
\alpha(gN,hN)=\hat{\alpha}(g,h).
\]
Moreover, we have a well-defined function 
    \[
	\cC_{K}:\{\theta\in\Irr(N)\mid K\leq \Stab_G(\theta)\}\longrightarrow \coho{2}(K/N),\qquad\cC_{K}(\theta)=[\alpha].
    \]
\end{thm}

Unlike in \cite{Rationality1}, where only $\cC_K$ was used, we will need to use
also $\cT_{L,K,\Gamma}(\tic{\theta})$ to establish Theorem~\ref{thm:Main-twist}. 

Throughout the current section, we let $G$ be an arbitrary profinite group. 
We say that two irreducible continuous complex representations $\rho,\sigma$ of $G$ are \emph{twist equivalent} if there exists
a one-dimensional representation $\psi$ of $G$ such that $\rho\otimes\psi\cong\sigma$.
This equivalence relation partitions the set of irreducible representations
of $G$ into \emph{twist isoclasses}. Let $\Lin(G)$ denote the set of characters in $\Irr(G)$ of degree one, that is, the linear continuous characters of $G$. We say that $\lambda,\delta\in\Irr(G)$ are \emph{twist equivalent} or lie in the same \emph{twist class} if $\lambda=\delta\psi$,
for some $\psi\in\Lin(G)$. Of course two representations are twist equivalent if and only if the characters they afford are. Note that twist equivalence preserves the dimension of representations, so we can speak of the dimension (degree) of a twist isoclass (twist class). 

If $H\leq G$ and $\psi:G\rightarrow\C^{\times}$ is a function (e.g.,
a degree one character), we will write $\psi|_{H}$ for $\Res_{H}^{G}(\psi)$.
We now define a twist equivalence relation for representations of
a subgroup of $G$, where the twisting is by degree one characters
which extend to $G$.

\begin{defn}
\label{def:G-twist}
	Let $H$ be a subgroup of $G$ and let $\rho$ and $\sigma$ be two
	irreducible representations of $H$. We say that $\rho$ and $\sigma$
	are \emph{$G$-twist equivalent,} and write $\rho\twist_{G}\sigma$,
	if there is a $\psi\in\Lin(G)$ such that
	\[
	\rho\otimes\psi|_{H}\cong\sigma.
	\]
	Similarly, two irreducible characters $\lambda,\delta\in\Irr(H)$
	are \emph{$G$-twist equivalent}, written $\lambda\twist_{G}\delta$,
	if $\lambda=\delta\psi|_{H}$, for some $\psi\in\Lin(G)$. 
\end{defn}
	For a character $\theta\in\Irr(H)$, we write $\tic{\theta}$ for the \emph{$G$-twist class} of $\theta$, that is,
	\[
	\tic{\theta}=\{\rho\in\Irr(H)\mid\rho\twist_{G}\theta\}
	\]
and we denote the set of such $G$-twist classes by $\twirr(H)$.
In particular, when $H=G$, $\twirr(G)$ is in bijection with the set of twist
isoclasses of $G$.

From now on, let $N$ be a normal subgroup of $G$ of finite index.
The conjugation action of $G$ on $\Irr(N)$ induces an action on
$\widetilde{\Irr}(N)$. Indeed, $g\cdot\tic{\theta}:=\tic{\leftexp{g}{\theta}\,}$
is well-defined because for any $\psi\in\Lin(G)$ and any $n\in N$,
we have
\[
\leftexp{g}{(\psi|_{N}\theta)}(n)=\leftexp{g}{\psi}(n)\leftexp{g}{\theta}(n)=\psi(n)\leftexp{g}{\theta}(n),
\]
and hence $\tic{\leftexp{g}{(\psi|_{N}\theta})}=\tic{\leftexp{g}{\theta}\,}$.

For any $\theta\in\Irr(N)$, define the stabiliser subgroups
\[
K_{\tic{\theta}}=\Stab_{G}(\theta),\qquad L_{\tic{\theta}}=\Stab_{G}(\tic{\theta}).
\]
Note that $K_{\tic{\theta}}$ only depends on the class $\tic{\theta}$
because $\Stab_{G}(\theta)=\Stab_{G}(\theta')$ for any $\theta'\in\tic{\theta}$.
It is clear that $K_{\tic{\theta}}\leq L_{\tic{\theta}}$, but in
fact we also have:
\begin{lem}
	The group $K_{\tic{\theta}}$ is normal in $L_{\tic{\theta}}$. 
\end{lem}

\begin{proof}
	Indeed, if $k\in K_{\tic{\theta}}$, $g\in L_{\tic{\theta}}$ and
	$x\in N$, then there exist some $\psi_{g}$, $\psi_{g^{-1}}\in\Lin(G)$
	such that 
	\[
	\leftexp{g}{\theta}(y)=\theta(y)\psi_{g}(y),\quad\leftexp{g^{-1}}{\theta}(y)=\theta(y)\psi_{g^{-1}}(y),\qquad\text{for all }y\in N.
	\]
	Thus
	\begin{align*}
	\leftexp{gkg^{-1}}{\theta}(x) & =\theta(gk^{-1}g^{-1}xgkg^{-1})=\theta(k^{-1}g^{-1}xgk)\psi_{g}(x)\\
	& =\theta(g^{-1}xg)\psi_{g}(x)=\theta(x)\psi_{g^{-1}}(x)\psi_{g}(x).
	\end{align*}
	But on the other hand, 
	\begin{align*}
	\theta(x) & =\leftexp{gg^{-1}}{\theta}(x)=\leftexp{g}{(\leftexp{g^{-1}}{\theta})}(x)=(\leftexp{g^{-1}}{\theta})(g^{-1}xg)\\
	& =\theta(g^{-1}xg)\psi_{g^{-1}}(g^{-1}xg)=\theta(x)\psi_{g}(x)\psi_{g^{-1}}(x),
	\end{align*}
	so $gkg^{-1}\in K_{\tic{\theta}}$.
\end{proof}

\subsection{Restriction and induction of twist classes}

Let $H$ be a group such that $N\leq H\leq G$. Let $\twirr(H\mid\tic{\theta})$
be the set of those $G$-twist classes $\tic{\lambda}\in\twirr(H)$
such that $\lambda\in\Irr(H\mid\psi|_{N}\theta)$, for some $\psi\in\Lin(G)$.
This is well-defined because $\lambda\in\Irr(H\mid\psi|_{N}\theta)$
if and only if $\psi'|_{H}\lambda\in\Irr(H\mid\psi'|_{N}\psi|_{N}\theta)$,
for all $\psi'\in\Lin(G)$.

The following is an immediate consequence of Clifford's theorem (see
\cite[(6.5)]{Isaacs}). Informally, it says that the $G$-twist classes
contained in the restriction to $N$ of $\tic{\rho}\in\twirr(H\mid\tic{\theta})$
are precisely the $H$-conjugates of $\tic{\theta}$.
\begin{lem}
	\label{lem:Cliffords-thm-twists}Let $\tic{\rho}\in\twirr(H\mid\tic{\theta})$.
	Then $\tic{\rho}\in\twirr(H\mid\leftexp{h}{\tic{\theta}})$, for any
	$h\in H$. Moreover, if $\tic{\rho}\in\twirr(H\mid\tic{\theta'})$,
	for some $\theta'\in\Irr(N)$, then there exists an $h\in H$ such
	that $\leftexp{h}{\tic{\theta}}=\tic{\theta'}$. 
\end{lem}

We now consider induction of twist classes.
Let $H$ and $H'$ be groups such that $K_{\tic{\theta}}\leq H\leq H'\leq G$.
Induction gives rise to a function
\begin{align*}
\twind_{H}^{H'}:\twirr(H\mid\tic{\theta}) & \longrightarrow\twirr(H'\mid\tic{\theta})\\
\tic{\lambda} & \longmapsto\tic{\Ind_{H}^{H'}\lambda},
\end{align*}
which is well-defined thanks to the formula $\Ind_{H}^{H'}(\psi|_{H}\lambda)=\psi|_{H'}\Ind_{H}^{H'}(\lambda)$,
for $\psi\in\Lin(G)$, and surjective thanks to standard Clifford
theory (see \cite[(6.11)(b)]{Isaacs}). However, unlike the classical
Clifford correspondence, where induction gives a bijection $\Irr(H\mid\theta)\rightarrow\Irr(H'\mid\theta)$,
the map $\twind_{H}^{H'}$ is not necessarily injective. Nevertheless, once we get up to the group $L_{\tic{\theta}}$, induction of twist classes behaves as in the classical Clifford correspondence, namely:

\begin{lem}
	\label{lem:Ind-is-bijective} The map $\twind_{L_{\tic{\theta}}}^{G}$
	is bijective. 
\end{lem}

\begin{proof}
	Let $\tic{\lambda},\tic{\lambda'}\in\twirr(L_{\tic{\theta}}\mid\tic{\theta})$
	such that $\Ind_{L_{\tic{\theta}}}^{G}\lambda\sim_{G}\Ind_{L_{\tic{\theta}}}^{G}\lambda'$.
	After multiplying by suitable degree one characters of $G$ restricted
	to $L_{\tic{\theta}}$ we may assume that both $\lambda$ and $\lambda'$
	lie above $\theta$. By hypothesis, there is a $\psi\in\Lin(G)$ such
	that 
	\[
	\Ind_{L_{\tic{\theta}}}^{G}\lambda=\psi\Ind_{L_{\tic{\theta}}}^{G}\lambda',
	\]
	so by Clifford's theorem there is a $g\in G$ such that $\leftexp{g}{\theta}=\psi|_{N}\theta$.
	Thus $g\in L_{\tic{\theta}}$ and $\lambda=\leftexp{g}{\lambda}$
	lies above $\psi|_{N}\theta$. Standard Clifford theory now implies
	that $\lambda=\psi|_{L_{\tic{\theta}}}\lambda'$ because $\lambda$ and $\psi|_{L_{\tic{\theta}}}\lambda'$ induce
	to the same irreducible character of $G$, both lie above $\psi|_{N}\theta$, 
	and $\Stab_{G}(\psi|_{N}\theta)=K_{\tic{\theta}}\leq L_{\tic{\theta}}$.
\end{proof}

\subsection{\label{subsec:The-function-bar-mu}The function $\bar{\mu}$ attached
	to a strong extension}

From now on, let $\theta\in\Irr(N)$ be fixed, let $L$ and $K$ groups
such that
\[
N\leq L\leq L_{\tic{\theta}}\qquad\text{and}\qquad N\leq K\leq K_{\tic{\theta}}.
\]
We assume that $L$ normalises $K$. 
The situations we will consider
where this is satisfied are when either $L=L_{\tic{\theta}}$ and
$K=L\cap K_{\tic{\theta}}$ or $L\leq L_{\tic{\theta}}$ and $K=K_{\tic{\theta}}$.

From now on, let $\hat{\theta}\in\PIrr_{\alpha}(K)$ be a strong extension
of $\theta$ to $K$ with factor set $\alpha$. For any $g\in L$,
we have
\begin{equation}
\leftexp{g}{\theta}=\theta\psi_{g}|_{N},\label{eq:g-theta-chi}
\end{equation}
for some $\psi_{g}\in\Lin(G)$. The conjugate projective character
$\leftexp{g}{\hat{\theta}}$ defined by $\leftexp{g}{\hat{\theta}}(x)=\hat{\theta}(g^{-1}xg)$
has factor set $\leftexp{g}{\alpha}$, where 
\[
\leftexp{g}{\alpha}(x,y)=\alpha(g^{-1}xg,g^{-1}yg)\qquad\text{for all }x,y\in K.
\]
Since both $\leftexp{g}{\hat{\theta}}$ and $\hat{\theta}\psi_{g}|_{K}$
are strong extensions of $\leftexp{g}{\theta}$, there exists a function
\[
\mu(g):K/N\rightarrow\C^{\times}
\]
(i.e., a function on $K$ constant on cosets of $N$) such that
\begin{equation}
\leftexp{g}{\hat{\theta}}=\hat{\theta}\psi_{g}|_{K}\cdot\mu(g).\label{eq:g-hat-theta-mu}
\end{equation}
Note that we may take $\mu(g)=\mu(gn)$, for any $n\in N$ because
$N$ fixes $\hat{\theta}$. Indeed, if $\Theta$ is a representation
affording $\theta$ and $\widehat{\Theta}$ is a projective representation
of $K$ affording $\hat{\theta}$, then, for any $n\in N$ and $x\in K$,
we have 

\begin{align}
\leftexp{n}{\hat{\theta}}(x) & =\Tr(\widehat{\Theta}(n^{-1}xn))=\Tr(\Theta(n^{-1})\widehat{\Theta}(x)\Theta(n))=\Tr(\widehat{\Theta}(x))=\hat{\theta}(x).\label{eq:N-fixes-theta-hat}
\end{align}
We will therefore henceforth write $\mu(gN)$ instead of $\mu(g)$. 

Using \eqref{eq:g-hat-theta-mu} and the fact that factor sets multiply
under tensor products of projective representations, we deduce that
the factor set of $\mu(gN)$ is $\leftexp{g}{\alpha}\alpha^{-1}$,
that is,
\[
\mu(gN)\in\PIrr_{\leftexp{g}{\alpha}\alpha^{-1}}(K/N).
\]

\begin{lem}
\label{lem:theta-hat-non-trivial-on-coset-Burnside}
	For every $x\in K$ there exists an $n\in N$ such that $\hat{\theta}(xn)\neq0$.
	Thus, for fixed $\theta$, the function $\mu(gN)$ is uniquely determined
	by $gN$, $\hat{\theta}$ and $\psi_{g}|_{K}$.
\end{lem}

\begin{proof}
	Let $\Theta$ be a representation affording $\theta$ and $\widehat{\Theta}$
	a projective representation of $K$ affording $\hat{\theta}$, so
	that $\hat{\theta}(xn)=\Tr(\widehat{\Theta}(xn))=\Tr(\widehat{\Theta}(x)\Theta(n))$.
	Assume that $\hat{\theta}(xn)=0$ for all $n\in N$. Then $\Tr(\widehat{\Theta}(x)\Theta(n))=0$
	for all $n\in N$, and by a theorem of Burnside (see \cite[(27.4)]{Curtis_Reiner})
	the values of $\Theta$ span the whole algebra $\M_{\theta(1)}(\C)$
	of matrices of size $\theta(1)$, so we have $\Tr(\widehat{\Theta}(x)A)=0$,
	for all $A\in\M_{\theta(1)}(\C)$. Since the trace form on $\M_{\theta(1)}(\C)$
	is non-degenerate, this implies that $\widehat{\Theta}(x)=0$, which is
	a contradiction. Thus $\hat{\theta}(xn)\neq0$ for some $n\in N$
	and
    	\[
	        \mu(gN)(xN)=\mu(gN)(xnN)=\leftexp{g}{\hat{\theta}}(xn)\hat{\theta}(xn)^{-1}\psi_{g}(xn)^{-1},
	    \]
	which proves the second assertion.
\end{proof}
We now consider how $\mu(gN)$ depends on $\psi_{g}|_{K}$. By \eqref{eq:g-theta-chi},
we have $\leftexp{g}{\theta}=\theta\psi_{g}|_{N}$. Let $\psi_{g}'\in\Lin(G)$
be such that $\leftexp{g}{\theta}=\theta\psi_{g}'|_{N}$. Then $\theta\, (\psi_{g}\psi_{g}'^{-1})|_{N} = \theta$
and since both $\hat{\theta}\,(\psi_{g}\psi_{g}'^{-1})|_{K}$ and $\hat{\theta}$
are strong extensions of $\theta$, we have 
\begin{equation}
\hat{\theta}\,(\psi_{g}\psi_{g}'^{-1})|_{K}=\hat{\theta}\cdot\nu_{g},\label{eq:theta-psi-nu_g}
\end{equation}
for some function $\nu_{g}:K/N\rightarrow\C^{\times}$. In fact, since
$\hat{\theta}\,(\psi_{g}\psi_{g}'^{-1})|_{K}$ and $\hat{\theta}$ have
the same factor set, $\nu_{g}$ has trivial factor set, that is, $\nu_{g}$
is a homomorphism.
Thus \eqref{eq:g-hat-theta-mu} can be written 
\begin{equation}
\leftexp{g}{\hat{\theta}}=\hat{\theta}\psi_{g}|_{K}\cdot\mu(gN)=\hat{\theta}\psi_{g}'|_{K}\cdot\mu(gN)\nu_{g}.\label{eq:mu-nu_g-ambiguity}
\end{equation}

\begin{defn}\label{def:Gamma}
Define the following subgroup of $\Lin(K/N)$.
	\[
\Gamma_{K,\tic{\theta}}=\{\nu\in\Lin(K/N)\mid\hat{\theta}\varepsilon|_{K}=\hat{\theta}\nu,\ \text{for some}\ \varepsilon\in\Lin(G)\}.
\]
(as usual, we denote by $\Lin(K/N)$ the subgroup of $\Lin(K)$ of characters which are trivial on $N$).
\end{defn}
In the present section, $K$ and $\tic{\theta}$ are fixed and we will simply write $\Gamma$ for $\Gamma_{K,\tic{\theta}}$.
Note that $\Gamma$ is independent of the choice of representative
$\theta$ of $\tic{\theta}$ and of the choice of strong extension
$\hat{\theta}$ of $\theta$. Indeed, if $\psi\in\Lin(G)$ and $\hat{\theta}'$
is a strong extension of $\theta\psi|_{N}$, then there exists a function
$\omega:K/N\rightarrow\C^{\times}$ such that
\[
\hat{\theta}\psi|_{K}\omega=\hat{\theta}'.
\]
Clearly $\hat{\theta}\varepsilon|_{K}=\hat{\theta}\nu$ holds for some $\varepsilon\in\Lin(G)$,
if and only if $\hat{\theta}\varepsilon|_{K}\psi|_{K}\omega = \hat{\theta}\nu\psi|_{K}\omega$, 
that is, by the equation above, if and only if $\hat{\theta}'\varepsilon|_{K}=\hat{\theta}'\nu$.

Moreover, for every $\nu\in\Gamma$ and $\psi_{g}$ as in \eqref{eq:g-hat-theta-mu}, if we
let $\varepsilon\in\Lin(G)$ be such that $\hat{\theta}\varepsilon|_{K}=\hat{\theta}\nu$,
we have that \eqref{eq:theta-psi-nu_g} holds with $\psi'_{g}=\varepsilon^{-1}\psi_{g}$
and $\nu_{g}=\nu$. Thus \eqref{eq:mu-nu_g-ambiguity} implies the
following.
\begin{lem}
	For any $g\in L$, the coset $\bar{\mu(gN)}:=\mu(gN)\Gamma$ is independent
	of the choice of $\psi_{g}|_{K}$ in \eqref{eq:g-hat-theta-mu}.
\end{lem}
In what follows, for a set $A$, we use the notation $\Func(A,\C^{\times})$ 
to denote the group of functions
$A\rightarrow\C^{\times}$ under pointwise multiplication.
The last lemma implies that, when $\theta$ is fixed, $gN$ and $\hat{\theta}$ uniquely
determine $\bar{\mu(gN)}$ and hence $\hat{\theta}$ uniquely determines
the function
\[
\bar{\mu}:L/N\longrightarrow F_{K}/\Gamma,\qquad g\longmapsto\bar{\mu(gN)},
\]
where
    \[
	F_{K}:=\Func(K/N,\C^{\times}).
    \]
We endow
the abelian group $F_{K}$ with the structure of $L/N$-module via
$gN\cdot f=\leftexp{g}{f}$, that is, $(gN\cdot f)(xN)=f(g^{-1}xgN)$ 
(this is well-defined because $K$ is normalised by $L$). 
Since $\hat{\theta}\varepsilon|_{K}=\hat{\theta}\nu$
implies, by conjugating both sides by $g$, that $\hat{\theta}\varepsilon|_{K}=\hat{\theta}\leftexp{g}{v}$,
$\Gamma$ is a submodule of $F_{K}$. Thus the quotient $F_{K}/\Gamma$
carries the corresponding $L/N$-module structure.

\subsection{\label{subsec:The-cohom-class-bar-mu}The cohomology class determined
	by $\bar{\mu}$}

We now consider how $\bar{\mu}$ depends on the choice of strong
extension $\hat{\theta}$ and the choice of representative $\theta\in\tic{\theta}$.
\begin{prop}
	\label{prop:function-mu-cohomology}Let $\theta\in\Irr(N)$ and let
	$\hat{\theta}$ be a strong extension of $\theta$ to $K$. The function
	$\bar{\mu}$ associated with $\hat{\theta}$ is an element of $\cocy{1}(L/N,F_{K}/\Gamma)$.
	The image $[\bar{\mu}]$ of $\bar{\mu}$ in $\coho{1}(L/N,F_{K}/\Gamma)$
	is uniquely determined by $\tic{\theta}$, that is, independent of
	the choice of strong extension $\hat{\theta}$ and independent of
	the choice of representative $\theta\in\tic{\theta}$.
\end{prop}

\begin{proof}
	For the first statement, we need to show that $\bar{\mu}$ is a crossed
	homomorphism, that is, that for all $g,g'\in L$, $\bar{\mu}(gg'N)=\bar{\mu}(gN)\leftexp{g}{\bar{\mu}(g'N)}$,
	or equivalently,
	\[
	\mu(gN)\leftexp{g}{\mu(g'N)}\Gamma=\mu(gg'N)\Gamma.
	\]
	By \eqref{eq:g-hat-theta-mu}, there exist some $\psi_{g},\psi_{g'}\in\Lin(G)$
	such that 
	\begin{align*}
	\leftexp{gg'}{\hat{\theta}} & =\leftexp{g}{(\hat{\theta}\psi_{g'}|_{K}\cdot\mu(g'N))}=\leftexp{g}{\hat{\theta}}\psi_{g'}|_{K}\cdot\leftexp{g}{\mu(g'N)}\\
	& =\hat{\theta}\psi_{g}|_{K}\cdot\mu(gN)\psi_{g'}|_{K}\cdot\leftexp{g}{\mu(g'N)}\\
	& =\hat{\theta}\,(\psi_{g}\psi_{g'})|_{K}\cdot\mu(gN)\leftexp{g}{\mu(g'N)}.
	\end{align*}
	On the other hand, for some $\psi_{gg'}\in\Lin(G)$ we have
	$\leftexp{gg'}{\hat{\theta}}=\hat{\theta}\psi_{gg'}|_{K}\cdot\mu(gg'N)$
	and hence
	\[
	\hat{\theta}\psi_{gg'}|_{K}\cdot\mu(gg'N)=\hat{\theta}\,(\psi_{g}\psi_{g'})|_{K}\cdot\mu(gN)\leftexp{g}{\mu(g'N)}.
	\]
	This is equivalent to
	\begin{equation}
	\hat{\theta}\,(\psi_{g}\psi_{g'})^{-1}|_{K}\psi_{gg'}|_{K}=\hat{\theta}\cdot\mu(gN)\leftexp{g}{\mu(g'N)}\mu(gg'N)^{-1},
	\end{equation}
	which implies that $\mu(gN)\leftexp{g}{\mu(g'N)}\mu(gg'N)^{-1}\in\Gamma$. Thus $\bar{\mu}$ is crossed homomorphism.
	
	For the second statement, let $\hat{\theta}'$ 
	be another strong extension of $\theta$ to $K$. Then there exists
	a function $\omega\in F_{K}$, such that $\hat{\theta}'=\hat{\theta}\omega$
	and hence, for any $g\in L$ there is a $\psi_{g}\in\Lin(G)$ such
	that
		\begin{equation}
			\leftexp{g}{\hat{\theta}'}
			\label{eq:another-strong-extn}
				= \hat{\theta}\psi_{g}|_{K}\cdot\mu(gN)\leftexp{g}{\omega}
				= \hat{\theta}' \psi_{g}|_{K} \cdot \mu(gN) \leftexp{g}{\omega} \omega^{-1}.
		\end{equation}
	The function 
	\[
	f:gN\mapsto\leftexp{g}{\omega}\omega^{-1}\Gamma=\leftexp{g}{\omega\Gamma}(\omega\Gamma)^{-1}
	\]
	lies in $\cobo{1}(L/N,F_{K}/\Gamma)$ and \eqref{eq:another-strong-extn}
	implies that $[\bar{\mu}]=[\bar{\mu}f]$. Hence both $\hat{\theta}$
	and $\hat{\theta}'$ determine the same element $[\bar{\mu}]\in \cocy{1}(L/N,F_{K}/\Gamma)/\cobo{1}(L/N,F_{K}/\Gamma)=\coho{1}(L/N,F_{K}/\Gamma)$.
	
	Finally, we need to show that $[\bar{\mu}]$ is independent of the
	choice of representative $\theta\in\tic{\theta}$. Let $\theta'\in\tic{\theta}$
	and let $\psi\in\Lin(G)$ be such that $\theta'=\theta\psi|_{N}$.
	Then $\hat{\theta}\psi|_{K}$ is a strong extension of $\theta'$.
	We want to compute $\bar{\mu}$ of $\theta'$ with respect to $\hat{\theta}\psi|_{K}$.
	For any $g\in L$, \eqref{eq:g-hat-theta-mu} yields
	\begin{align*}
	\leftexp{g}{(\hat{\theta}\psi|_{K})} & =\leftexp{g}{\hat{\theta}}\psi|_{K}=\hat{\theta}\psi_{g}|_{K}\psi|_{K}\cdot\mu(gN)\\
	& =(\hat{\theta}\psi|_{K})\psi_{g}|_{K}\cdot\mu(gN).
	\end{align*}
	Thus $\theta$ and $\theta'$ give rise to the same element
	$\bar{\mu}$, with respect to the strong extensions $\hat{\theta}$ and $\hat{\theta}\psi|_{K}$, respectively.
	By the independence of $[\bar{\mu}]$ on the choice of strong extension
	proved above, we conclude that $\theta$ and $\theta'$ give rise to the same element $[\bar{\mu}]$.
\end{proof}

\subsection{The function $\cT_{L,K,\Gamma}$}
\label{sec:The_function_cT_L_K_Gamma}
So far we have associated $[\bar{\mu}]\in \coho{1}(L/N,F_{K}/\Gamma)$
with a fixed class $\tic{\theta}\in\twirr(N)$. We now consider the
situation when $\tic{\theta}$ varies, but with $K$
and $L$ fixed. 

Let $L$ and $K$ be as in the beginning of Section~\ref{subsec:The-function-bar-mu}
and let $\Gamma$ be any subgroup of $\Lin(K/N)$.
Define 
\begin{align*}
\twirr_{L,K,\Gamma}(N) &= \{\tic{\theta}\in\twirr(N)\mid L=L_{\tic{\theta}},\, K=K_{\tic{\theta}},\, \Gamma=\Gamma_{K,\tic{\theta}}\},\\
\twirr^{\leq}_{L,K,\Gamma}(N) &= \{\tic{\theta}\in\twirr(N)\mid L\leq L_{\tic{\theta}},\,K\leq K_{\tic{\theta}},\,\Gamma=\Gamma_{K,\tic{\theta}}\},
\end{align*}
where $\Gamma_{K,\tic{\theta}}$ is as in Definition~\ref{def:Gamma}.
Note that $\twirr_{L,K,\Gamma}(N)$ may well be empty for some $L,K,\Gamma$.
Proposition~\ref{prop:function-mu-cohomology} implies that we may
define the following function
\[
\cT_{L,K,\Gamma}:\twirr^{\leq}_{L,K,\Gamma}(N) \longrightarrow \coho{1}(L/N,F_{K}/\Gamma),\qquad \cT_{L,K,\Gamma}(\tic{\theta})=[\bar{\mu}].
\]

\section{Reduction to pro-$p$ Sylow subgroups}\label{sec:Reduction-to-pro-p_twisted_case}

Throughout the present section, $G$ will denote a profinite group and $N$ a 
finite index normal pro-$p$ subgroup of $G$. Let $N\leq K\trianglelefteq L\leq G$
and $\Gamma$ be an arbitrary subgroup of $\Lin(K/N)$. 
For any prime $q$ dividing $|L/N|$, let $L_q$ be a subgroup of $L$ such that $L_q/N$ is a Sylow $q$-subgroup of $L/N$. 
Similarly, let $K_q$ be a subgroup of $K$ such that $K_q/N$ is a Sylow $q$-subgroup of $K/N$. We may and will assume that 
    \[
	K_p = K \cap L_p.
    \]

Let $H\leq G$ be a group that fixes $\theta$. We note that the function $\cC_H$ defined in \cite[\thmCliffordMap]{Rationality1}
induces a function on twist classes. Let $\theta'\in\tic{\theta}$ so that $\theta'=\theta\psi|_{N}$ for
some $\psi\in\Lin(G)$. Let $\hat{\theta}$ be a strong extension of $\theta$. Then $\hat{\theta}\psi|_{H}$ is a strong 
extension of $\theta'$ with the same factor set as that of $\hat{\theta}$, and thus 
\[
\cC_{H}(\theta)=\cC_{H}(\theta').
\]
This shows that the function $\cC_H$ is constant on the twist class $\tic{\theta}$, so $\cC_{H}(\tic{\theta})$ is well-defined.

The goal of this section (Proposition~\ref{prop:red_coeff_to_Sylow}) is to show that the invariant 
$\cT_{L,K,\Gamma}(\tic{\theta})$ attached to a twist class $\tic{\theta} \in \twirr^{\leq}_{L,K,\Gamma}(N)$ 
is determined by $\cC_{K_p}(\tic{\theta})$ together with $\cT_{L_p,K_p,\Gamma_p}(\tic{\theta})$, where $\Gamma_p$ is
the image of $\Gamma$ under the map defined by restricting homomorphisms of $K/N$ to $K_{p}/N$.

Let $q$ be a prime. We denote the $q$-primary component of a torsion abelian group $M$ by $M_{(q)}$ and write $m_{(q)}$ for the $q$-part of an element $m\in M$. 
\begin{lem}\label{lem:Z-BU_H-finite}
	Let $G$ be a finite group of order $m$ and $M$ be an abelian group
	(written multiplicatively) on which $G$ acts. Assume that $M$ is
	finitely divisible in the sense that for any positive integer $n$ and $a\in M$,
	there is a finite but non-zero number of elements $x\in M$ such that
	$x^{n}=a$. Then, for any integer $i\geq1$, we have
	\[
	\cocy{i}(G,M)=\cobo{i}(G,M)U^{i},
	\]
	where $U^{i}=\{\alpha\in \cocy{i}(G,M)\mid\alpha^{m}=1\}$. Moreover,
	$\coho{i}(G,M)$ is finite.
\end{lem}

\begin{proof}
	We first prove that $\cobo{i}(G,M)$ is divisible. A function $\beta:G^{i}\rightarrow M$
	is in $\cobo{i}(G,M)$ if and only if it is of the form
	\begin{align*}
	\beta(g_{1},\dots,g_{i}) & =\leftexp{g_{1}}{f(g_{2},\dots,g_{i})}f(g_{1},\dots,g_{n})^{(-1)^{i}}\\
	& \hphantom{{}={}} \prod_{j=1}^{i-1}f(g_{1},\dots,g_{j-1},g_{j}g_{j+1},\dots,g_{i})^{(-1)^{j}}
	\end{align*}
	for some function $f:G^{i-1}\rightarrow M$, where $G^{0}:=\{1\}$
	(see, e.g., \cite[VII.3]{serre}). Let $\beta$ and $f$ be such that
	this holds. Since $M$ is divisible, there exists, for any positive integer $n$,
	a function $\widetilde{f}\in G^{i-1}\rightarrow M$ such that $\widetilde{f}^{n}=f$.
	Thus
	\begin{multline*}
	\beta(g_{1},\dots,g_{i})=\leftexp{g_{1}}{\widetilde{f}(g_{2},\dots,g_{i})^{n}}\widetilde{f}(g_{1},\dots,g_{n})^{(-1)^{i}n}\\
	\prod_{j=1}^{i-1}\widetilde{f}(g_{1},\dots,g_{j-1},g_{j}g_{j+1},\dots,g_{i})^{(-1)^{j}n}\\
	=\Big(\leftexp{g_{1}}{\widetilde{f}(g_{2},\dots,g_{i})}\widetilde{f}(g_{1},\dots,g_{n})^{(-1)^{i}}\prod_{j=1}^{i-1}\widetilde{f}(g_{1},\dots,g_{j-1},g_{j}g_{j+1},\dots,g_{i})^{(-1)^{j}}\Big)^{n},
	\end{multline*}
	so $\beta=\gamma^{n}$ for some $\gamma\in \cobo{i}(G,M)$.
	
	Let $\alpha\in \cocy{i}(G,M)$. By \cite[\lemBasicCoho\,(i)]{Rationality1}
	we have $\alpha^{m}\in \cobo{i}(G,M)$. Since $\cobo{i}(G,M)$ is divisible,
	there is a $\beta \in \cobo{i}(G,M)$ such that $\alpha^{m}=\beta^{m}$, and hence
	$\alpha \beta^{-1}\in U^{i}$. We thus have $\alpha\in \cobo{i}(G,M)U^{i}$
	and since $\alpha$ was arbitrary, $\cocy{i}(G,M)=\cobo{i}(G,M)U^{i}$.
	
	Now, every element in $U^{i}$ is a function $G^{i}\rightarrow\{a\in M\mid a^{m}=1\}$.
	The codomain is a finite set since $M$ is finitely divisible, so
	$U^{i}$ is finite and hence $\coho{i}(G,M)$ is finite, since it embeds in $U^i$.
\end{proof}

\subsection{Reduction of the parameter $L$}
In this subsection we will show that for $\tic{\theta}, \tic{\theta'} \in \twirr_{L,K,\Gamma}(N)$, 
	\[
	\cT_{L_p,K,\Gamma}(\tic{\theta}) = \cT_{L_p,K,\Gamma}(\tic{\theta'})
		\Longrightarrow
			\cT_{L,K,\Gamma}(\tic{\theta}) = \cT_{L,K,\Gamma}(\tic{\theta'}).
	\]
	(see Proposition~\ref{prop:red_L_to_Sylow}).
In order to prove this we need two lemmas. 
 Let $W\leq \C^{\times}$ be the group of roots of unity. Let 
$f \in \Func(K, \C^\times)$ and $\ell$ a prime. As in \cite[\subsecGoodBases]{Rationality1}, we will fix a homomorphism 
$\pi_{\ell}:  \C^{\times} \rightarrow W_{(\ell)}$ and denote $\pi_{\ell}\circ f$ by $f_{(\ell)}$. Then for any $g\in G$ normalising $K$ and any prime $\ell$, we have 
$\leftexp{g}{(f_{(\ell)})} = (\leftexp{g}{f})_{(\ell)}$. Recall also the notation $F_K=\Func(K/N,\C^{\times})$.
\begin{lem}
	\label{lem:pi_ell_mu-cobound}
	Let $q$ be a prime dividing $\lvert L/N \rvert$ and $\mu \in \cocy{1}(L_q/N, F_K)$.
	Suppose that there exists a function $\omega_q: K/N \to W_{(q)}$ such that for all $g \in L_q$,
	\[
	\mu(gN)_{(q)} = \leftexp{g}{\omega_q} \omega_q^{-1}\cdot\nu_g,
	\]
	for some $\nu_g\in\Gamma$.
	Then $\bar{\mu} \in \cobo{1}(L_q/N, F_K/\Gamma)$.
\end{lem}
\begin{proof}
	Let $U = \lbrace \alpha \in \cocy{1}(L_q/N, F_K) \mid \alpha^{\lvert L_q/N \rvert} = 1 \rbrace$.
	Then, as $\C^{\times}$, and hence $F_K$, is finitely divisible, Lemma~\ref{lem:Z-BU_H-finite} implies that
	\[
	\cocy{1}(L_q/N, F_K) = U \cobo{1}(L_q/N, F_K).
	\]
	Thus there is a $\tau \in U$ and a function $\omega: K /N \to \C^{\times}$ 
	such that, for $g \in L_q$, $\mu(gN) = \tau(gN) \leftexp{g}{\omega}\omega^{-1}$.
	This implies that $ \mu(gN)_{(q)} = \tau(gN)_{(q)}\leftexp{g}{(\omega_{(q)})} {\omega}^{-1}_{(q)}$ and hence 
        \[
	        \tau(gN)_{(q)} = \mu(gN)_{(q)} (\leftexp{g}{(\omega_{(q)})} {\omega}^{-1}_{(q)})^{-1}.
        \]
	Combined with the equation
	$\mu(gN)_{(q)} = \leftexp{g}{\omega_q} \omega_q^{-1}\nu_g$ this implies that
	\[
	    \tau(gN)_{(q)} = \leftexp{g}{\omega_q} \omega_q^{-1} (\leftexp{g}{(\omega_{(q)})} {\omega}^{-1}_{(q)})^{-1} \nu_g 
				= \leftexp{g}{(\omega_q\omega_{(q)}^{-1})} (\omega_q\omega_{(q)}^{-1})^{-1} \nu_g.
	\]
	Since $\tau(gN)$ has values in $W_{(q)}$, we have $\tau(gN)=\tau(gN)_{(q)}$ and thus, for all $g\in L_q$,
	\begin{align*}
	    \mu(gN)     & = \tau(gN) \leftexp{g}{\omega}\omega^{-1} = 
				\leftexp{g}{(\omega_q\omega_{(q)}^{-1})} (\omega_q\omega_{(q)}^{-1})^{-1} \leftexp{g}{\omega}\omega^{-1} \nu_g \\
			& = \leftexp{g}{(\omega_q\omega_{(q)}^{-1}\omega)} (\omega_q\omega_{(q)}^{-1}\omega)^{-1}\nu_g.
	\end{align*}
	Hence the function $\bar{\mu} : L_q/N\rightarrow F_K/\Gamma$ is an element in $\cobo{1}(L_q/N, F_K/\Gamma)$.
\end{proof}
\begin{lem}\label{lem:triv_l_not_q}
	Let $\tic{\theta}\in \twirr_{L,K,\Gamma}(N)$. For every prime $q$ dividing $\lvert L/N \rvert$ such that 
	$q\neq p$, we have $\cT_{L_q,K,\Gamma}(\tic{\theta}) = 1$.
\end{lem}
\begin{proof}   
	Let $\Theta$ be a representation of $N$ affording the character $\theta$. 
	Then, for $g \in L_q$, there is a $\psi_g\in \Lin(G)$, 
	$P_g \in \GL_{\theta(1)}(\C)$ and $\mu \in \cocy{1}(L_q/N, F_K)$ such that 
		\begin{equation}
		\label{eq:mu_for_Theta}
			\leftexp{g}{\widehat{\Theta}} =  \,P_g^{-1} \widehat{\Theta} P_g \cdot \psi_g\lvert_K\cdot\mu(gN).
		\end{equation}
	
	By definition we have $\cT_{L_q,K,\Gamma}(\tic{\theta}) = [\bar{\mu}]$, so by Lemma~\ref{lem:pi_ell_mu-cobound} 
	it suffices to show that there is a function $\omega_q:\nolinebreak K/N \to W_{(q)}$ such that for all $g \in L_q$,
	\begin{equation}
	\label{eq:triviality_q}
	\mu(gN)_{(q)} = \leftexp{g}{\omega_q}\omega_q^{-1} \nu_g,
	\end{equation}
	for some $\nu_g\in \Gamma$.
	To prove this, let $\xi = \det  \circ \,\widehat{\Theta}$ so that $\xi \in \Func(K,\C^{\times})$ 
	(note that the use of this function is the reason we cannot work only with projective characters in this proof). 
	Then, by equation \eqref{eq:mu_for_Theta}, 
		\[
			\leftexp{g}{\xi}{\xi}^{-1} = \mu(gN)^{\theta(1)} (\psi_g \lvert_K)^{\theta(1)}.
		\]
	and hence 
	\begin{equation}
	\label{eq:xi}
	\leftexp{g}{\xi}_{(q)}\xi^{-1} = 
	\mu(gN)_{(q)}^{\theta(1)} (\psi_g\lvert_K)_{(q)}^{\theta(1)}.
	\end{equation}
	Now, $\theta(1)$ is a power of $p$ so it is coprime to $q$. This means that raising 
	to the power of $\theta(1)$ is an automorphism of $W_{(q)}$. Therefore there exists a unique function 
	$\omega_q: K\rightarrow W_{(q)}$ such that $\omega_q^{\theta(1)} = \xi_{(q)}$ and  \eqref{eq:xi} implies that 
	\[
	\leftexp{g}{\omega_q}\omega_q^{-1} = \mu(gN)_{(q)} (\psi_g\lvert_K)_{(q)}.
	\]    
	
	We finish the proof by showing that the last equality implies equation \eqref{eq:triviality_q}, that is, 
	that $(\psi_g\lvert_K)_{(q)}\in \Gamma$ and that $\omega_q$ is constant on cosets of $N$.
	First observe that $(\psi_g\lvert_N)_{(q)}$ 
	is a homomorphism from a pro-$p$ group
	to the $q$-group $W_{(q)}$, so it must be trivial. 
	By the definition of $\Gamma_{K, \tic{\theta}}$ (Definition~\ref{def:Gamma}), it therefore follows that 
	$(\psi_g\lvert_K)_{(q)} \in \Gamma_{K, \tic{\theta}} = \Gamma$. 
	
	It remains to show that $\omega_q$ is constant on the cosets of $N$ in $K$. Indeed, let $t \in K$ 
	and $n \in N$. Then $\widehat{\Theta}(tn) = \widehat{\Theta}(t) \Theta(n)$, so
	$\xi(tn) = \xi(t) \xi(n)$ and hence  $\xi_{(q)}(tn) = \xi_{(q)}(t) \xi_{(q)}(n)$. 
	Since $\xi \lvert_N = \det \circ \Theta$ is a homomorphism from the pro-$p$ group $N$ to $\C^{\times}$,  
	$\xi_{(q)}$ is trivial on $N$. It follows that  $\xi_{(q)}(tn) = \xi_{(q)}(t)$, so 
	$\omega_q(tn)^{\theta(1)}=\xi_{(q)}(tn)=\xi_{(q)}(t)=\omega_q(t)^{\theta(1)}$ and 
	thus $\omega_q(tn) = \omega_q(t)$.
\end{proof}
\begin{prop}
	\label{prop:red_L_to_Sylow}
	Let $\tic{\theta}, \tic{\theta'} \in \twirr_{L,K,\Gamma}(N)$. Then 
	\[
	\cT_{L_p,K,\Gamma}(\tic{\theta}) = \cT_{L_p,K,\Gamma}(\tic{\theta'})
	\Longrightarrow
	\cT_{L,K,\Gamma}(\tic{\theta}) = \cT_{L,K,\Gamma}(\tic{\theta'}).
	\]
\end{prop}
\begin{proof}
By \cite[\lemBasicCoho]{Rationality1}, $\coho{1}(L/N,F_{K}/\Gamma)$ is a torsion abelian group so we can write
\[
\cT_{L,K,\Gamma}(\tic{\theta}) = 
\prod_q
\cT_{L,K,\Gamma}(\tic{\theta})_{(q)},
\]
where $q$ runs through the primes dividing $\lvert L/N \rvert$.
      Let $q$ be a prime dividing $|L/N|$. By \cite[\lemBasicCoho]{Rationality1}, 
	    \[
		\res_{q}:\coho{1}(L/N,F_K/\Gamma)_{(q)}\longrightarrow\coho{1}(L_{q}/N,F_K/\Gamma)
	    \]
	is injective. We claim that 
	\begin{equation}
	\res_{q}(\cT_{L,K,\Gamma}(\tic{\theta})_{(q)})=\cT_{L_{q},K,\Gamma}(\tic{\theta})
	\label{eq:res(T_L-K-Gamma(theta))}
	\end{equation}
	(and similarly for $\theta'$). Indeed, letting $\mu$ be such that $\cT_{L,K,\Gamma}(\tic{\theta}) = [\bar{\mu}]$, 
	we have
		\[
			\res_{L/N,L_{q}/N}(\cT_{L,K,\Gamma}(\tic{\theta})) = [\bar{\mu}\lvert_{L_q}] = \cT_{L_{q},K,\Gamma}(\tic{\theta}),
		\]
	where the second equality holds by the definition of $\cT_{L_q,K,\Gamma}(\tic{\theta})$.
	Furthermore, since $\coho{1}(L_{q}/N,F_K/\Gamma)$ is a $q$-group, the homomorphism 
	$\res_{L/N,L_{q}/N}$ is trivial on $\coho{1}(L/N,F_K/\Gamma)_{(\ell)}$, for any prime
	$\ell\neq q$. Thus,  
	    \[
		\res_{L/N,L_{q}/N}(\cT_{L,K,\Gamma}(\tic{\theta}))
			=\res_{L/N,L_{q}/N}(\cT_{L,K,\Gamma}(\tic{\theta})_{(q)})=\res_{q}(\cT_{L,K,\Gamma}(\tic{\theta})_{(q)}),
	    \]
	proving \eqref{eq:res(T_L-K-Gamma(theta))}.

Now, if $q\neq p$, Lemma~\ref{lem:triv_l_not_q} implies that $\cT_{L_{q},K,\Gamma}(\tic{\theta})=1$ and
	by \eqref{eq:res(T_L-K-Gamma(theta))} we obtain $\res_{q}(\cT_{L,K,\Gamma}(\tic{\theta})_{(q)})=1$,
	whence $\cT_{L,K,\Gamma}(\tic{\theta})_{(q)}=1$ (by the injectivity of $\res_{q}$). We must therefore have 
	$\cT_{L,K,\Gamma}(\tic{\theta})=\cT_{L,K,\Gamma}(\tic{\theta})_{(p)}$, and since
	$\theta$ was arbitrary, we similarly have $\cT_{L,K,\Gamma}(\tic{\theta'})=\cT_{L,K,\Gamma}(\tic{\theta'})_{(p)}$.
	Applying \eqref{eq:res(T_L-K-Gamma(theta))} for $q=p$, we get
	    \[
		\res_{p}(\cT_{L,K,\Gamma}(\tic{\theta})_{(p)})=\cT_{L_p,K,\Gamma}(\tic{\theta})
			     = \cT_{L_p,K,\Gamma}(\tic{\theta'})=\res_{p}(\cT_{L,K,\Gamma}(\tic{\theta'})_{(p)}),
	    \]
	and we conclude that $\cT_{L,K,\Gamma}(\tic{\theta})_{(p)}=\cT_{L,K,\Gamma}(\tic{\theta'})_{(p)}$,
	and thus $\cT_{L,K,\Gamma}(\tic{\theta})=\cT_{L,K,\Gamma}(\tic{\theta'})$. 
\end{proof}
\subsection{Reduction of the coefficient module}
We have shown that $\cT_{L, K, \Gamma}(\tic{\theta})$ is determined by 
$\cT_{L_{p}, K, \Gamma}(\tic{\theta})\in \coho{1}(L_p/N,F_{K}/\Gamma)$. We will now further show that the latter 
is determined by an element in $\coho{1}(L_p/N,F_{K_p}/\Gamma_p)$ where 
    \[
        \Gamma_p = \{ \nu\lvert_\Kp \mid \nu \in \Gamma \}.
    \]
\subsubsection{Reduction of the parameter $\Gamma$}
We start by investigating the structure of $\Gamma_{K, \tic{\theta}}$.
\begin{defn}
We define
	\[	
		\Gamma_{K}^{0} = \{ \nu \in \Lin(K/N) \mid \nu = \varepsilon\lvert_K, \text{ for some } \varepsilon\in \Lin(G),\, \nu_{(p)} = 1\}.
	\]
\end{defn}
\begin{lem}
\label{lem:struct_Gamma}
Let $\theta \in \Irr(N)$ such that $K\leq \Stab_G(\theta)$. 
\begin{enumerate}
	\item  \label{lem:struct_Gamma_i}
	Then $\Gamma_{K, \tic{\theta}}$ splits as the 
	(internal) direct product
			\[
				\Gamma_{K, \tic{\theta}} = \Gamma_{K}^{0}\, \big(\Gamma_{K, \tic{\theta}}\big)_{(p)}
			\]
	where $\big(\Gamma_{K, \tic{\theta}}\big)_{(p)} = \{ \nu_{(p)} \mid \nu \in \Gamma \}$.
	\item \label{lem:struct_Gamma_ii}
		Moreover, let $\rho: \Lin(K) \to \Lin(\Kp)$ be the 
		homomorphism of abelian groups induced by restricting maps 
		on $K$ to maps on $\Kp$. Then $\Gamma_K^0 \leq \Ker \rho$ and $\rho$ restricted to 
		$\big(\Gamma_{K, \tic{\theta}}\big)_{(p)}$ is injective with image $\Gamma_{\Kp, \tic{\theta}}$.
	\end{enumerate}
\end{lem}
	\begin{proof}
	We prove the first statement. Let $\hat{\theta}$ be a strong extension of $\theta$. 
	First $\Gamma_{K}^{0} \leq \Gamma_{\Kp, \tic{\theta}}$. Indeed, if $\nu \in \Gamma_{K}^{0}$, then $\nu\lvert_N = 1$ and 
	$\nu = \varepsilon\lvert_K$ for some $\varepsilon \in \Lin(G)$. We have $\varepsilon\lvert_N = 1$, so 
	$\theta \varepsilon\lvert_N = \theta$. Thus $\hat{\theta} \nu = \hat{\theta} \varepsilon_{K}$, so 
	$\nu \in \Gamma$. 
	Second, 
	$\Gamma_{K}^{0} \cap \big(\Gamma_{\Kp, \tic{\theta}}\big)_{(p)} = 1$ because, by definition, $\nu_{(p)} = 1$ for 
	all  $\nu \in \Gamma_{K}^{0}$.
	Let now  $\nu \in \Gamma$, so that 
		\[
			\nu = \prod_{q} \nu_{(q)}
		\]
	where the product runs over primes $q \mid \lvert K : N \rvert$. We prove that for 
	$q \neq p$, $\nu_{(q)} \in \Gamma_{K}^{0}$. Fix $q \neq p$. Since $(\nu_{(q)})_{(p)} = 1$,
	it suffices to show that $\nu_{(q)}$ is the restriction of a character in $\Lin(G)$. 
	In order to do so, let $\Theta$ be representation
	affording $\theta$ and let $\widehat{\Theta}$ be a strong extension of $\Theta$ to $K$. By 
	definition of $\Gamma_{\Kp, \tic{\theta}}$ we have that there are $P \in \GL_{\theta(1)}(\C)$ and $\varepsilon \in \Lin(G)$
	such that 
		\[
			\varepsilon\lvert_{K} \cdot\,  \widehat{\Theta} = \nu \cdot\,  P^{-1} \widehat{\Theta} P.
		\]
	This implies that
		\[
			\varepsilon\lvert_{K}^{\theta(1)} \cdot\,  (\det \circ \, \widehat{\Theta})= \nu^{\theta(1)} \cdot\,  (\det \circ\, \widehat{\Theta}).
		\]
	Hence $(\varepsilon\lvert_{K})^{\theta(1)} = \nu^{\theta(1)}$ and so 
	$((\varepsilon\lvert_{K})^{\theta(1)})_{(q)} = (\nu^{\theta(1)})_{(q)}$. The decomposition of 
	a root of unity into roots of unity of prime power order is multiplicative, hence
	$(\varepsilon_{(q)}\lvert_{K})^{\theta(1)} = \nu_{(q)}^{\theta(1)}$. 
	Since $q \neq p$ and $\theta(1)$ is a power of $p$, it follows that  
	$\varepsilon_{(q)}\lvert_{K}= \nu_{(q)}$.\par
	We prove the second part. Clearly $\Gamma_{K}^{0} \leq \Ker \rho$ because $\Kp$ is a pro-$p$ group. 
	Moreover, every 
	homomorphism $K/N \rightarrow W_{(p)}$ factors through
		\[
			\frac{K}{[K,K]N} =  \prod_{q} \frac{K_q [K,K]N}{[K,K]N}.
		\]
	where the product runs over primes $q \mid \lvert K : N \rvert$.
	For $q \neq p$, there are no non-trivial homomorphisms 
	$K_q[K, K]N/[K,K]N \rightarrow W_{(p)}$. Thus $\rho$ is injective on $\big(\Gamma_{\Kp, \tic{\theta}}\big)_{(p)}$ and
	we need only prove the statement about its image. To this end, let $\widehat{\Theta}_p = 
	\widehat{\Theta}\lvert_{\Kp}$ and let $\nu_p \in \Gamma_{\Kp, \tic{\theta}}$. Then 
	there are $P \in \GL_{\theta(1)}(\C)$ and $\varepsilon \in \Lin(G)$ such that
		\begin{equation*}
			\varepsilon\lvert_{\Kp} \cdot\, \widehat{\Theta}_p = \nu_p P^{-1} \widehat{\Theta}_p P.
		\end{equation*}
	Restricting both sides of the last equality to $N$ we have that 
	$\varepsilon\lvert_{N} \cdot\,  \Theta = P^{-1} \Theta P$, so $\varepsilon\lvert_{K} \cdot\,  \widehat{\Theta}$ and 
	$P^{-1} \widehat{\Theta} P$ are both strong extensions of $\varepsilon\lvert_{N} \cdot\,  \Theta$. Thus there
	is a scalar function $\nu:K/N \to \C^{\times}$ such that
		\[
			\nu \cdot I_{\theta(1)}= \varepsilon\lvert_{K} \cdot\,  \widehat{\Theta} P^{-1} \widehat{\Theta}^{-1} P.
		\]
	By its very definition, $\nu \in \Gamma$ and so $\nu_{(p)} \in \big(\Gamma_{\Kp, \tic{\theta}}\big)_{(p)}$. 
	This is enough to conclude, because $\Kp$ is a pro-$p$ group and hence 
	we have $\nu_{(p)}\lvert_{\Kp} = \nu\lvert_{\Kp} = \nu_p$.	
	\end{proof}
The following consequence of the structure of $\Gamma_{K, \tic{\theta}}$ will achieve the 
goal of this subsection and will also be key to producing a first order formula 
for the predicate $\Gamma_{K, \tic{\theta}} = \Gamma$ in Section~\ref{sec:rationality_partial_tw}.
\begin{prop}
\label{prop:red_Gamma}
Let $\tic{\theta} \in  \twirr(N)$ such that $K \leq K_{\tic{\theta}}$. 
Assume there exists $\tic{\theta}' \in  \twirr(N)$ such that $\Gamma_{K, \tic{\theta}'} = \Gamma$.
Then $\Gamma_{K,\tic{\theta}} = \Gamma$ if and only if $\Gamma_{\Kp, \tic{\theta}} = \Gamma_p$.
\end{prop}
\begin{proof}
		Part \ref{lem:struct_Gamma_ii}
		of Lemma~\ref{lem:struct_Gamma} gives that 
			\[
				\Gamma_{\Kp, \tic{\theta}} = \Gamma_p \iff \big(\Gamma_{\Kp, \tic{\theta}}\big)_{(p)} = \Gamma_{(p)}.
			\] 
		By part \ref{lem:struct_Gamma_i} of Lemma~\ref{lem:struct_Gamma}, 
		the latter is equivalent to $\Gamma_{\Kp, \tic{\theta}}  = \big(\Gamma_{K}^{0} \Gamma_{\Kp, \tic{\theta}}\big)_{(p)} 
		=  \Gamma_{K}^{0}\Gamma_{(p)} = \Gamma$.
	\end{proof}
\subsubsection{Reduction of the parameter $K$}
\label{subsec:red_coeff_mod} 
In what follows, we let $\theta \in \twirr_{L,K,\Gamma}(N)$. Proposition~\ref{prop:red_Gamma} implies that 
    \[
	\twirr_{L,K,\Gamma}(N) \subseteq \twirr_{L_p, K_p, \Gamma_p}^{\leq}(N)
    \]
(see Section~\ref{sec:The_function_cT_L_K_Gamma} for the definitions of these sets)
and therefore, for $\tic{\theta} \in \twirr_{L,K,\Gamma}(N)$, 
the element $\cT_{L_{p}, K_{p}, \Gamma_p} (\tic{\theta})\in \coho{1}(L_p/N,F_{K_p}/\Gamma_p)$ is well-defined. 
The following proposition shows that $\cT_{L_{p}, K, \Gamma}(\tic{\theta})$
is determined by $\cT_{L_{p}, K_p, \Gamma_p}(\tic{\theta})$.
\begin{prop}
	\label{prop:red_coeff_to_Sylow}
	Let $\tic{\theta}, \tic{\theta'} \in \twirr_{L,K,\Gamma}(N)$ and assume that 
	$\cC_{K_{p}}(\tic{\theta}) = \cC_{K_{p}}(\tic{\theta'})$. Then
	\[
	 \cT_{L_{p}, K_{p}, \Gamma_p} (\tic{\theta})= \cT_{L_{p}, K_{p}, \Gamma_p}(\tic{\theta'})
	 \Longrightarrow
	\cT_{L, K, \Gamma} (\tic{\theta})= \cT_{L, K, \Gamma}(\tic{\theta'}).
	\]
\end{prop}
\begin{proof}
	By Proposition~\ref{prop:red_L_to_Sylow} it suffices to prove that 
		\begin{equation*}
		\cT_{L_p, K_p, \Gamma_p} (\tic{\theta})= \cT_{L_p, K_p, \Gamma_p}(\tic{\theta'})
			\Longrightarrow
				\cT_{L_p, K, \Gamma} (\tic{\theta}) = 
				\cT_{L_p, K, \Gamma}(\tic{\theta'}).
		\end{equation*} 
	 By \cite[\lemJaikinsProp]{Rationality1} our hypothesis $\cC_{K_{p}}(\tic{\theta}) = \cC_{K_{p}}(\tic{\theta'})$ implies that 
	$\cC_{K}(\tic{\theta}) = \cC_{K}(\tic{\theta'})$. 	
	Therefore, by \cite[\lemSameFs]{Rationality1},
	there exist strong extensions  $\hat{\theta}$ and $\hat{\theta}'$ with the same factor set.
	Thus there exist $\mu, \mu' \in  \cocy{1}(L_{p}/N, F_K)$, such that for all $g \in L_{p}$ there are $\psi_g, \psi'_g\in \Lin(G)$ 
	with 
		\[
			\leftexp{g}{\hat{\theta}}   = \hat{\theta} \psi_g\lvert_K \cdot  \mu(gN) \qquad \text{and} \qquad
			\leftexp{g}{\hat{\theta}'}  = \hat{\theta}' \psi'_g\lvert_K \cdot  \mu'(gN).
		\]
	Since $\hat{\theta}$ and $\hat{\theta}'$ have the same factor set, we have
		\begin{equation}
	   	\label{eq:mu-mu_prime-hom}
			\mu(gN)^{-1} \mu'(gN) \in \Lin(K/N)=\Hom(K/N, \C^{\times}).
		\end{equation}
	Assume now that 
	$\cT_{L_{p}, K_{p}, \Gamma_p} (\tic{\theta})= \cT_{L_{p}, K_{p}, \Gamma_p}(\tic{\theta'})$.
	Then there is a function $\eta: K_{p}/N \rightarrow \C^{\times}$ such that, for any $g\in L_{p}$,
		\[
			(\mu(gN)^{-1}\lvert_{K_{p}}) (\mu'(gN)\lvert_{K_{p}}) \Gamma_p
				= \leftexp{g}{\eta}\eta^{-1} \Gamma_p.
		\]
		
	Changing $\psi_g$ and $\psi'_g$ if necessary we may without loss of generality assume that
	    \begin{equation}
	    \label{eq:triv_K_p}
		\mu(gN)^{-1}\lvert_{K_{p}} \mu'(gN)\lvert_{K_{p}} = \leftexp{g}{\eta}\eta^{-1}.
	    \end{equation}
	
	By \eqref{eq:mu-mu_prime-hom} and \eqref{eq:triv_K_p}, 
	$\leftexp{g}{\eta}\eta^{-1}$ is the restriction of an element in $\Lin(K/N)$, hence
	it is trivial on $(K_{p}\cap [K,K])N$. 
	By the second isomorphism theorem, $\eta$ defines a function on $K_{p}[K,K]N$, which is constant on cosets of $[K,K]N$. 
	By abuse of notation, we denote this function by $\eta$ as well. 
	The finite abelian group $K/([K,K]N)$ factors as 
	    \begin{equation}
	    \label{eq:splitting_K_KKN}
		\frac{K}{[K,K]N} = \frac{K_{p}[K,K]N}{[K,K]N} \prod_q
		\frac{K_q[K,K]N}{[K,K]N},
	    \end{equation}
	where $q$ runs through the primes dividing  $|K: K_{p}|$.
	Extend $\eta$ to the function $\hat{\eta}: K/([K,K]N) \rightarrow \C^{\times}$ such that $\hat{\eta} = 1$ on $\frac{K_{q}[K,K]N}{[K,K]N}$ for every $q\neq p$. 
	By \eqref{eq:mu-mu_prime-hom} and \eqref{eq:triv_K_p}, the function $\leftexp{g}{\hat{\eta}}\hat{\eta}^{-1}$ is 
	a homomorphism on $\frac{K_{p}[K,K]N}{[K,K]N}$ 
	and it is trivial on each of the other factors on the 
	right-hand side of \eqref{eq:splitting_K_KKN}. Thus $\leftexp{g}{\hat{\eta}}\hat{\eta}^{-1}$ is a homomorphism $K/N \rightarrow \C^{\times}$. 
	Therefore 
	    \[
		(\leftexp{g}{\hat{\eta}}\hat{\eta}^{-1})_{(p)} = \leftexp{g}{\hat{\eta}_{(p)}}\hat{\eta}_{(p)}^{-1}
	    \]
	is a homomorphism $K/N \rightarrow W_{(p)}$ and, by \eqref{eq:triv_K_p},
	    \[
		\mu(gN)_{(p)}^{-1}\lvert_{K_{p}}  \mu'(gN)_{(p)}\lvert_{K_{p}} 
		    = \leftexp{g}{\eta_{(p)}} \eta_{(p)}^{-1} 
			=\leftexp{g}{\hat{\eta}_{(p)}} \hat{\eta}_{(p)}^{-1}|_{K_p}.
	    \] 
	Since $\mu(gN)_{(p)}^{-1} \mu'(gN)_{(p)}$ is a homomorphism from $K/([K,K]N)$ to the $p$-group $W_{(p)}$, it is trivial on every factor $\frac{K_{q}[K,K]N}{[K,K]N}$, $q\neq p$, in \eqref{eq:splitting_K_KKN}. We therefore have 
	    \[
		\mu(gN)_{(p)}^{-1}  \mu'(gN)_{(p)} = \leftexp{g}{\hat{\eta}_{(p)}}\hat{\eta}_{(p)}^{-1},
	    \]
	for any $g\in L_p$, so by Lemma~\ref{lem:pi_ell_mu-cobound},
	    \[
		\bar{\mu^{-1} \mu'} \in \cobo{1}(L_p/N,K/\Gamma),
	    \]
	that is,
	    \[
		\cT_{L_p, K, \Gamma} (\tic{\theta}) = 
			    \cT_{L_p, K, \Gamma}(\tic{\theta'}).
	    \]
\end{proof}

\section{Reduction to the partial twist zeta series}
\label{sec:Reduction partial twist}

\subsection{Dirichlet polynomials for twist character triples}
If $\tic{\theta}\in \twirr(N)$ and $N \leq L\leq L_{\tic{\theta}}$ we will call $(L,N,\tic{\theta})$ a \emph{$G$-twist character triple}. Given such a triple, define the Dirichlet polynomial
\[
\widetilde{f}_{(L,N,\tic{\theta})}(s)=
\sum_{\tic{\lambda}\in\twirr(L \mid\tic{\theta})}
\left(\frac{\lambda(1)}{\theta(1)}\right)^{-s}.
\]
The goal of this subsection is to prove 
Proposition~\ref{prop:Analogue-Jaikin}, that is, that the invariants $\cC_{K_p}(\tic{\theta})$ and $\cT_{L_p,K_p,\Gamma_p}(\tic{\theta})$ associated with $\tic{\theta}$ determine $\widetilde{f}_{(L,N,\tic{\theta})}(s)$. This will be used in the proof of Proposition~\ref{prop:partial-Main-twist} in the following subsection.

We start with some straightforward generalisations to projective characters 
of some of the notation and formalism regarding induction and restriction.
Let $H$ be a profinite group and let $\alpha \in \cocy{2}(H)$. 
	If $M$ is a normal subgroup of $H$ of finite index and $\theta \in \PIrr_{\alpha_M}(M)$ we define
	\[
	\PIrr_{\alpha}(H \mid \theta) = 
	\bigl\{ \pi \in \PIrr_{\alpha}(H) \mid 
	\bigl< \Res^H_M \pi, \theta \bigr> \neq 0 \bigr\}.
	\]

From now on, let $\alpha\in \cocy{2}(L/N)$.
\begin{lem}
	\label{lem:isaacs_6_11_proj}
	Let $\theta \in \PIrr_{\alpha_N}(N)$. Assume that $L\cap K_{\theta}\leq K$ (i.e., $\Stab_L(\theta) \leq K$) and that $\alpha_K$ is $L$-invariant
	Then we have a bijection
	\[
	\Ind_{K,\alpha}^L : \PIrr_{\alpha_K }(K \mid \theta) 
	\longrightarrow \PIrr_{\alpha}(L \mid \theta).
	\]
\end{lem}
\begin{proof}
	The proof of \cite[Theorem 6.11]{Isaacs} transfers, mutatis mutandis, to the present 
	situation as Frobenius reciprocity and Clifford's theorem hold in the more general 
	context of projective representations (see, e.g.,
	\cite[Theorem~10.1]{Karpilovsky2} for the latter).
\end{proof}
We generalise the $G$-twist equivalence relation $\Gtwist$ on $\Irr(L)$ to $\PIrr_{\alpha}(L)$ in the obvious way, that is, for $\pi_1, \pi_2 \in \PIrr_{\alpha}(L)$,
\[
\pi_1 \Gtwist \pi_2    \iff    \pi_1 =  \pi_2 \psi\lvert_L, \quad \text{for some}\ \psi\in \Lin(G).
\]
For a projective character $\theta$ denote its $G$-twist class
by $\tic{\theta}$ and we denote the set of $G$-twist classes in $\PIrr_{\alpha}(L)$ by 
$\ptwirr_{\alpha}(L)$.
Moreover, if $\theta \in \PIrr_{\alpha_N}(N)$, we define 
\[
\ptwirr_{\alpha}(L \mid \tic{\theta})
\] 
as the set of those $G$-twist classes 
$\tic{\pi}\in\ptwirr_{\alpha}(L)$
such that $\pi\in\PIrr_{\alpha}(L\mid\theta \psi\lvert_{N})$, for some $\psi\in\Lin(G)$.
%
The $G$-twist equivalence relation is compatible 
with $\Ind_{K,\alpha}^L$ in the sense that for any $\lambda \in \PIrr_{\alpha_K}(K)$ and $\psi\in \Lin(G)$, we have $\Ind_{K,\alpha}^L(\lambda)\psi = \Ind_{K,\alpha}^L(\lambda \psi|_H)$. This follows immediately from the character formula for induced projective characters; see \cite[Chapter~1, Proposition~9.1\,(i)]{Karpilovsky3}.
Thus, if $\alpha_K$ is $L$-invariant, there is a function
\begin{equation}
    \label{eq:twist_iso_ind}
\twind_{K,\alpha}^L : \ptwirr_{\alpha_K}(K \mid \tic{\theta}) 
\longrightarrow \ptwirr_{\alpha}(L \mid \tic{\theta})
\end{equation}
sending $\tic{\pi} \in \ptwirr_{\alpha_K}(K \mid \tic{\theta})$ to the $G$-twist 
class of $\Ind_{K,\alpha}^L \pi$. 

The following lemma is a straightforward 
application of Mackey's intertwining number formula for projective characters \cite[Ch.~1, Theorem~8.6]{Karpilovsky3}. 
\begin{lem}     
	\label{lem:same_ind}
	Let $\pi_1, \pi_2 \in \PIrr_{\alpha_K}(K)$ and assume that $\alpha_K$ is $L$-invariant.
	Then
       \[
	\twind_{K,\alpha}^L \tic{\pi}_1 = \twind_{K,\alpha}^L \tic{\pi}_2 
	\iff \exists\,\, g \in L\  (\pi_1 \Gtwist \leftexp{g}{\pi_2}).
	\]
\end{lem}
\begin{prop}
	\label{prop:Analogue-Jaikin}
	Let $\tic{\theta}, \tic{\theta'} \in \twirr_{L,K,\Gamma}(N)$ and assume that 
	$\cC_{K_{p}}(\tic{\theta}) = \cC_{K_{p}}(\tic{\theta'})$ and $\cT_{L_{p}, K_{p}, \Gamma_p} (\tic{\theta})= \cT_{L_{p}, K_{p}, \Gamma_p}(\tic{\theta'})$. Then
	\[
	 \widetilde{f}_{(L,N,\tic{\theta})}(s)=\widetilde{f}_{(L,N,\tic{\theta'})}(s).
	\] 
\end{prop}
\begin{proof}
	\newcommand{\comp}{\mathrm{\sigma}}
	\newcommand{\precomp}{\mathrm{\sigma_0}}
	We prove this by constructing a bijection 
	\[
	\comp : \twirr(L \mid\tic{\theta}) \longrightarrow \twirr(L \mid\tic{\theta'})
	\]
	such that for all $\tic{\lambda} \in \twirr(L \mid\tic{\theta})$ we have
	\[
	\frac{\lambda(1)}{\theta(1)} = \frac{\lambda'(1)}{\theta'(1)}
	\]
	for $\lambda' \in \comp(\tic{\lambda})$.\par
	
	As	$\cC_{K_{p}}(\tic{\theta}) = \cC_{K_{p}}(\tic{\theta'})$, we have 	$\cC_{K}(\tic{\theta}) = \cC_{K}(\tic{\theta'})$  by \cite[\lemJaikinsProp]{Rationality1}, so \cite[\lemSameFs]{Rationality1}  implies that there are strong extensions $\hat{\theta}$ and $\hat{\theta}'$ of
	$\theta$ and $\theta'$, respectively, with the same factor set, say $\alpha$.
	Suppose that 
	$\cT_{L_{p}, K_{p}, \Gamma_p} (\tic{\theta})= \cT_{L_{p}, K_{p}, \Gamma_p}(\tic{\theta'})$.
	Then by Proposition~\ref{prop:red_coeff_to_Sylow}, there are cocycles $\mu, \mu' \in \cocy{1}(L/N, F_{K})$, 
	such that for any $g \in L$ there exist $\psi_g, \psi'_g \in \Lin(G)$ with 
	\begin{align*}
	\leftexp{g}{\hat{\theta}}       &= \hat{\theta} \psi_g\lvert_K \cdot \mu(gN),\\
	\leftexp{g}{\hat{\theta}'}      &= \hat{\theta}' \psi'_g\lvert_K \cdot \mu'(gN),\\
	\mu(gN) \Gamma	  &= \mu'(gN) \leftexp{g}{\eta}\eta^{-1}\Gamma,
	\end{align*}
	for some function $\eta: K/N \to \C^{\times}$. By changing $\psi_g$ and $\psi'_g$ if necessary,
	we may assume without loss of generality that $\mu(gN) = \mu'(gN) \leftexp{g}{\eta}\eta^{-1}$.
	This gives, in particular, that 
	\begin{equation}
	\label{eq:mod_mu}
	\leftexp{g}{(\eta\hat{\theta}')} 
	= \big( \eta\hat{\theta}'\big) \psi'_g\lvert_K  \cdot \mu(gN).
	\end{equation}
	
	Let $\omega: L/N \to \C^{\times}$ be a function extending $\eta$ and let 
	$\delta \in \cobo{2}(L/N)$ 	
	be its factor set. Clearly, the restriction
	$\delta_K$ equals the factor set of $\eta$.
	Let $\tic{\lambda} \in \twirr(L \mid \tic{\theta})$ and let 
	$\rho \in \Irr(K \mid \theta)$ be such that 
	\[
	\tic{\lambda} = \tic{\Ind}_K^L \tic{\rho}.
	\]
	By \cite[\lemCliffordExt]{Rationality1} we have $\rho = \hat{\theta} \cdot \pi$ for a unique $\pi\in \PIrr_{\alpha^{-1}}(K/N)$.     
	We define 
	\[
	\precomp(\tic{\lambda}) = 
	\twind_{K,\delta}^L \big(\tic{\eta\hat{\theta}' \cdot \pi}\big).
	\]
	Note that $\hat{\theta}' \cdot \pi \in \Irr(K)$, so $\eta\hat{\theta}' \cdot \pi \in \PIrr_{\delta_K}(K)$ 
	and since $\eta$ has constant value $\eta(1)$ on $N$, we have 
	$\precomp(\tic{\lambda})\in \ptwirr_{\delta}(L\mid \eta(1)\tic{\theta}')$.
	
	We will show that 
	\[
	\precomp: \twirr(L \mid \tic{\theta}) \longrightarrow \ptwirr_{\delta}(L\mid \eta(1)\tic{\theta}')
	\]
	is a well-defined bijection.\par 
	First, to prove that $\precomp$ is well-defined we need to prove that 
	$\precomp(\tic{\lambda})$ is independent of the choice of the $G$-twist class $\tic{\rho}$ inducing 
	to $\tic{\lambda}$. To this end, suppose that $\rho^* \in \Irr(K \mid \theta)$ is another character 
	such that $\tic{\lambda} = \tic{\Ind}_K^L \tic{\rho^*}$ and let  
	$\rho^* = \hat{\theta} \cdot \pi^*$ with $\pi^*\in \PIrr_{\alpha^{-1}}(K/N)$. 
	The relation $\tic{\Ind}_K^L \tic{\rho} = \tic{\Ind}_K^L \tic{\rho^*}$ implies 
	(by Mackey's induction-restriction theorem for ordinary characters) that there 
	is a $g \in L$ such that $\leftexp{g}{(\hat{\theta}\cdot \pi)} \Gtwist \hat{\theta}\cdot \pi^*$.
	Moreover, we have
	\begin{equation}
	      \label{eq:pre-mod_mu}
	\begin{split}
	\leftexp{g}{(\hat{\theta}\cdot\pi)} 
	& =\leftexp{g}{\hat{\theta}} \cdot \leftexp{g}{\pi} 
	=\big( \hat{\theta}\psi_{g}|_{K}\cdot\mu(gN)\big)\cdot\leftexp{g}{\pi}\\
	&=\hat{\theta}\cdot \big(\leftexp{g}{\pi}\psi_{g}|_{K}\cdot\mu(gN)\big),
	\end{split}
	\end{equation}
	and thus 
	$ \pi^* \Gtwist \leftexp{g}{\pi}\mu(gN)$.
	On the other hand, by equation \eqref{eq:mod_mu},
	\begin{equation}
	\label{eq:mod_mu-pi}
	\leftexp{g}{(\eta \hat{\theta}'\cdot\pi)} 
	= \eta \hat{\theta}'  \cdot \big(\leftexp{g}{\pi} \psi_{g}'|_{K}\cdot\mu(gN)\big),
	\end{equation}
	so
	    \[ 
		\leftexp{g}{(\eta\hat{\theta}'\cdot \pi)} \Gtwist \eta\hat{\theta}'\cdot \pi^*.
	    \]
	As $\hat{\theta}$ and $\hat{\theta'}$ have the same factor set, 
	$\mu(gN)$ and $\mu'(gN)$ have the same factor set, so
	$\mu^{-1}(gN)\mu'(gN)$, and thus $\leftexp{g}{\eta}\eta^{-1}$, is a homomorphism for all $g \in L$. Hence the factor set $\delta_K$ of $\eta$ is $L$-invariant. 
	We can thus apply Lemma~\ref{lem:same_ind} to obtain that 
	 \[
	\twind_{K,\delta}^L \big(\tic{\eta\hat{\theta}' \cdot \pi}\big)
	=  \twind_{K,\delta}^L \big(\tic{\eta\hat{\theta}' \cdot \pi^*}\big),
	\]
	that is, $\precomp$ is well-defined. 
	
	Similarly, we can prove that $\precomp$ is injective. Indeed, if $\precomp(\tic{\lambda})=\precomp(\tic{\lambda}^*)$, with 
	$\tic{\lambda} = \tic{\Ind}_K^L \tic{\rho}$, $\rho=\hat{\theta}\cdot \pi$ and $\tic{\lambda}^* = \tic{\Ind}_K^L \tic{\rho^*}$, $\rho^*=\hat{\theta}\cdot \pi^*$, 
	then Lemma~\ref{lem:same_ind} implies that there is a $g\in L$ such that
	$\leftexp{g}{(\eta\hat{\theta}'\cdot \pi)} \Gtwist \eta\hat{\theta}'\cdot \pi^*$, so by \eqref{eq:mod_mu-pi}, $\pi^* \Gtwist \leftexp{g}{\pi}\mu(gN)$, hence by 
	\eqref{eq:pre-mod_mu} we get $\leftexp{g}{(\hat{\theta}\cdot \pi)} \Gtwist \hat{\theta}\cdot \pi^*$, which by Lemma~\ref{lem:same_ind} implies that $\tic{\lambda}=\tic{\lambda}^*$.
	
	The surjectivity part of Lemma~\ref{lem:isaacs_6_11_proj} implies that 
	the function in equation~\eqref{eq:twist_iso_ind} 
	is surjective. Thus the function $\precomp$ is surjective and hence bijective.\par
	We now define
	\[
	\comp(\tic{\lambda}) = \omega^{-1} \cdot\precomp(\tic{\lambda}), 
	\qquad \text{ for } \tic{\lambda} \in \twirr(L \mid \tic{\theta}). 
	\]
	Multiplying by $\omega^{-1}$ is clearly a bijection $\ptwirr_{\delta}(L\mid \eta(1)\tic{\theta'})\rightarrow \twirr(L\mid\tic{\theta'})$ so $\comp$ is a bijection
	$\twirr(L \mid \nolinebreak \tic{\theta}) \to \twirr(L \mid\tic{\theta'})$.
	Moreover, for all $\tic{\lambda} \in \twirr(L \mid \tic{\theta})$ with $\tic{\lambda} = \tic{\Ind}_K^L \tic{\rho}$, $\rho=\hat{\theta}\cdot \pi$
	and $\lambda' \in \comp(\tic{\lambda})$, we have
	\begin{align*}
	\lambda(1)      & = \lvert L : K \rvert \theta(1)\pi(1),\\
	\lambda'(1)     &= \omega(1)^{-1} \lvert L : K \rvert \,\omega(1)\theta'(1)\pi(1) = \lvert L : K \rvert\,  \theta'(1)\pi(1).
	\end{align*}
	This concludes the proof.
\end{proof}
\subsection{Reduction of Theorem~\ref{thm:Main-twist} to the partial twist zeta series}
From now on, let $G$ be a twist-rigid compact $p$-adic analytic
group.  Note that $G$ is allowed to be FAb here and that in this case we may well have $Z_G(s) \neq \widetilde{Z}_G(s)$. Let $N$ be a normal
open pro-$p$ subgroup of $G$.
	      
As in Sections~\ref{sec:Twist-iso-Clifford} and \ref{sec:Reduction-to-pro-p_twisted_case} we write
\[
K_{\tic{\theta}}=\Stab_{G}(\theta),\qquad L_{\tic{\theta}}=\Stab_{G}(\tic{\theta}),
\]
for any $\theta\in\Irr(N)$.
For $K$, $L$, $\Gamma$, $K_p$ and $L_p$ as in Section~\ref{sec:Reduction-to-pro-p_twisted_case} and 
for any $c\in\coho{2}(K_{p}/N)$ (we follow the convention in \cite{Rationality1} to drop the trivial coefficient module $\C^{\times}$ from cohomology groups as well as from cocycle and coboundary groups) 
and $c'\in \coho{1}(L_{p}/N,M_{K_{p}}/\Gamma_p)$,
define
\begin{align*}
\twirrparams  = \{\tic{\theta}\in\twirr_{L,K,\Gamma}(N)\mid 
 \cC_{K_{p}}(\tic{\theta})=c,\ \cT_{L_{p}, K_{p}, \Gamma_p}(\tic{\theta})=c'\}.
\end{align*}
In analogy with the partial representation zeta series defined earlier,
we introduce the \emph{partial twist zeta series} 
\[
\tpartial{N;L, K, \Gamma}{c, c'} =\sum_{\tic{\theta}\in\twirr_{L, K, \Gamma}^{c,c'}(N)}\theta(1)^{-s}.
\]
Note that $\tpartial{N;L, K, \Gamma}{c, c'}=0$ unless there is a $\theta \in \Irr(N)$ such that 
$K=K_{\tic{\theta}}$, $L=L_{\tic{\theta}}$ and $\Gamma_{K,\tic{\theta}} = \Gamma$. Note also that
\[
\widetilde{Z}_{N}(s)=
\sum_{\substack{	N\leq K\leq L\leq G \\ 
		\Gamma \leq \Lin(K/N)}}
\sum_{\substack{	c\in\coho{2}(K_{p}/N)\\
		c'\in \coho{1}(L_{p}/N,F_{K_{p}}/\Gamma_p)}}
\tpartial{N;L, K, \Gamma}{c, c'}.
\]

Let $\cS$
denote the set of subgroups $K\leq G$ such that $N\leq K$ and $K=K_{\tic{\theta}}$
for some $\theta\in\Irr(N)$.
Similarly, let $\widetilde{\cS}$ denote the set of subgroups $L\leq G$ such that $N\leq L$ and $L=L_{\tic{\theta}}$, for some $\tic{\theta}\in\twirr(N)$.
For $K \in \cS$, let $\cG(K)$ be the set of 
subgroups of $\Gamma \leq \Lin(K/N)$ such that $\Gamma=\Gamma_{K,\tic{\theta}}$,
for some $\tic{\theta}\in\twirr(N)$ such that $K\leq K_{\tic{\theta}}$. 
\begin{prop}\label{prop:partial-Main-twist} 
	Suppose that $\tpartial{N;L, K, \Gamma}{c, c'}$ is rational in $p^{-s}$,
	for every $L\in\widetilde{\cS}$, $K\in\cS$, $\Gamma \in \cG(K)$, 
	$c\in\coho{2}(K_{p}/N)$ and $c'\in \coho{1}(L_{p}/N,F_{K_p}/\Gamma)$. Then Theorem~\ref{thm:Main-twist} holds.
\end{prop}

\begin{proof}
	By Lemma~\ref{lem:Cliffords-thm-twists}, for every $\tic{\rho}\in\twirr(G)$, there are exactly $|G:L_{\tic{\theta}}|$
	distinct $G$-twist classes $\tic{\theta}\in\twirr(N)$ such that
	$\tic{\rho}\in\twirr(G\mid\tic{\theta})$. Thus
	\[
	\widetilde{Z}_{G}(s)    =       \sum_{\tic{\rho}\in\twirr(G)}\rho(1)^{-s} = 
	\sum_{\tic{\theta} \in \twirr(N)}   |G:L_{\tic{\theta}}|^{-1}\sum_{\tic{\rho}\in\twirr(G\mid\tic{\theta})}\rho(1)^{-s}.
	\]
	By Lemma~\ref{lem:Ind-is-bijective},
	induction of $G$-twist classes from $\twirr(L_{\tic{\theta}}\mid\tic{\theta})$
	to $\twirr(G\mid\tic{\theta})$ is a bijective map. Therefore,
	\[
	\sum_{\tic{\rho}\in\twirr(G\mid\tic{\theta})} \rho(1)^{-s} = 
	\sum_{\tic{\lambda}\in\twirr(L_{\tic{\theta}}\mid\tic{\theta})} (\lambda(1)\cdot|G:L_{\tic{\theta}}|)^{-s},
	\]
	and so
	\begin{align*}
	\widetilde{Z}_{G}(s) & =	\sum_{\tic{\theta}\in\twirr(N)}|G:L_{\tic{\theta}}|^{-s-1}
	\sum_{\tic{\lambda} \in \twirr(L_{\tic{\theta}}\mid\tic{\theta})} 
	\theta(1)^{-s} \left(\frac{\lambda(1)}{\theta(1)}\right)^{-s}\\
	& = \sum_{\tic{\theta}\in\twirr(N)}
	|G:L_{\tic{\theta}}|^{-s-1} \theta(1)^{-s} \widetilde{f}_{(L_{\tic{\theta}},N,\tic{\theta})}(s)\\
	& = \sum_{\substack{L\in\widetilde{\cS}}} |G:L|^{-s-1} 
	\sum_{\substack{\tic{\theta}\in\twirr(N)\\L_{\tic{\theta}}=L}}
	\theta(1)^{-s} \widetilde{f}_{(L,N,\tic{\theta})}(s).
	\end{align*}
	If $\tic{\theta},\tic{\theta'}\in\twirrparams$, then
	$\cC_{K_{p}}(\tic{\theta})=\cC_{K_{p}}(\tic{\theta'})$
	and $\cT_{L_{p}, K_{p}, \Gamma_p}(\tic{\theta})=\cT_{L_{p}, K_{p}, \Gamma_p}(\tic{\theta'})$.
	Thus, by Proposition~\ref{prop:Analogue-Jaikin}, we have
	$\widetilde{f}_{(L,N,\tic{\theta})}(s)=\widetilde{f}_{(L,N,\tic{\theta'})}(s)$. 
	By the above expression for $\widetilde{Z}_{G}(s)$, 
	we can therefore write
	\[
	\widetilde{Z}_{G}(s)=\sum_{\substack{   L\in\widetilde{\cS}\\
			K\in\cS\\
			\Gamma \in \cG(K)}}
	|G:L|^{-s-1}\sum_{\substack{    c\in\coho{2}(K_{p}/N)\\
			c'\in \coho{1}(L_{p}/N,F_K/\Gamma)}}
	\widetilde{f}_{L,K,\Gamma}^{c,c'}(s)\tpartial{N;L, K, \Gamma}{c, c'}
	\]
	where 
	$\widetilde{f}_{L,K,\Gamma}^{c,c'}(s) := \widetilde{f}_{(L,N,\tic{\theta})}(s)$ for some
	(equivalently, any) $G$-twist character triple $(L,N,\tic{\theta})$ such that $\tic{\theta} \in \twirrparams$.
	
	From the assumption that $\tpartial{N;L, K, \Gamma}{c, c'}$ is rational
	in $p^{-s}$, it now follows that $\widetilde{Z}_{G}(s)$, and hence $\widetilde{\zeta}_{G}(s)$,
	is virtually rational. Moreover, if $G$ is pro-$p$, then $|G:L|$
	is a power of $p$ for any subgroup $L$, and likewise $\lambda(1)$
	is a power of $p$, so $\widetilde{f}_{(L,N,\tic{\theta})}(s)$ is a polynomial
	in $p^{-s}$. Thus, when $G$ is pro-$p$, $\widetilde{Z}_{G}(s)$, and
	hence $\widetilde{\zeta}_{G}(s)$, is rational in $p^{-s}$.
\end{proof}
\section{Rationality of the partial twist zeta series}
\label{sec:rationality_partial_tw}
The groups $G$, $N$, $K$, $L$, $\Gamma$, $K_p$ and $L_p$ are as in the previous two sections. 
We shall show that, for each $c\in\coho{2}(K_p/N)$ and $c'\in \coho{1}(L_{p}/N,F_{K_{p}}/\Gamma_p)$, the set of twist classes $\twirrparams$  
is in bijection with a set of equivalence classes under a definable equivalence relation. 
We deduce from this that each partial twist zeta series is rational in $p^{-s}$ and hence prove 
Theorem~\ref{thm:Main-twist}.  Fix $c\in\coho{2}(K_p/N)$ and $c'\in \coho{1}(L_{p}/N,F_{K_{p}}/\Gamma_p)$ throughout the 
section. In order to use Proposition~\ref{prop:red_Gamma} we assume in this section that 
$\twirrparams \neq \emptyset$.\par
\subsection{Reduction of the predicate $\cT_{\Lp,\Kp,\Gamma_p}(\tic{\theta}) = c'$ to degree one characters}
In this section we reduce the computation of $\cT_{\Lp,\Kp,\Gammap}(\tic{\theta})$ to a statement 
on a degree one character. Namely, if $c'\in \coho{1}{(\Lp/N, F_{\Kp}/\Gamma_p)}$, our goal is to 
express  $\cT_{\Lp,\Kp,\Gammap}(\tic{\theta}) = c'$ with a statement involving only 
elements of $N$, conjugation by elements of $G$, and a linear character of a finite-index subgroup of $N$. \par
We shall make use of the following notation from \cite[\secRedDegOne]{Rationality1}. We define
	\[
		\cH(K_p)=\{H\leq K_p\mid H\text{ open in }K_p,\,K_p=HN\},
	\]
and let $X_K$ be the set of pairs $(H,\chi)$ with $H\in\cH(K_p)$, where:
\begin{enumerate}
	\item $(H,N\cap H,\chi)$ is a character triple.
	\item $\chi$ is of degree one,
	\item $\Ind_{N\cap H}^{N}\chi\in \Irr_K(N)$. 
\end{enumerate}\par
We fix a pair $(H, \chi) \in X_K$ such that $\theta = \Ind_{N \cap H}^{N}(\chi)$. 
We define 
	\[
		\indexKpN = \lvert \Kp : N \rvert,\, 	\indexLpN = \lvert \Lp : N \rvert,\text{ and } 	\indexGN = \lvert G : N \rvert.
	\] 
Let $( y_1, \dots, y_\indexGN)$ be a left transversal for $N$ in $G$ with $y _1 = 1$, and such that
$y_1, \dots, y_\indexKpN \in \Kp$ and $y_{\indexKpN + 1},\dots, y_{\indexLpN} \in \Lp$. 
Note that $\Kp = H N$ implies that 
there exist elements $t_1,\dots,t_{\indexPN}\in N$ such that $(y_1t_1,\dots,y_{\indexKpN}t_{\indexKpN})$ 
is a left transversal for $N\cap H$ in $H$.\par
In view of proving definability in $\man$ we shall express conjugation by elements in $G$ in 
terms of the associated automorphism of $N$. In order to do so, 
for $i \in \{ 1, \dots \indexGN\}$, we define $\varphi_i$ 
to be the automorphism of $N$ sending $n$ to $y_i n y_i^{-1}$ and set
	\begin{align*}
		C_\Kp 		&= \{ \varphi_i \mid i = 1, \dots, \indexKpN \},		&
		C_{\Lp} 		&= \{ \varphi_i \mid i  = 1, \dots, \indexLpN \},		&
		C_G			&= \{ \varphi_i \mid i  = 1, \dots, \indexGN \}.
	\end{align*}
Finally we need to express how conjugating by a coset representative acts on the 
other coset representatives. To this end, we define
\begin{gather*}
\kappa: \{ 1,\dots, \indexGN\} \times \lbrace 1, \dots, \indexGN \rbrace \longrightarrow\lbrace 1,\dots,\indexGN \rbrace, \text{ and } d_{ij}\in N \text{ by}\\
y_i^{-1}y_j y_i = y_{\kappa(i,j)}d_{ij},
\end{gather*}
for all $i\in \{ 1,\dots, \indexLpN\}$ and $j \in \{ 1,\dots, \indexKpN  \}$. The choice of using right conjugation in the 
definition of $\kappa$ 
is more natural here, as it will only be needed to simplify expressions in the argument of $\chi$.\par
Let $\theta$ be a representative of $\tic{\theta}$, we need to choose a strong extension of 
$\theta$. By \cite[\propRedLinear]{Rationality1}, this may be 
done by inducing a strong extension $\hat{\chi}$ of $\chi$ from $H$ to $\Kp$. Since the definition of 
$\cT_{\Lp,\Kp,\Gammap}(\tic{\theta})$ is independent of the choice of strong extension (of $\theta$), we may 
assume without loss of generality that $\hat{\chi}$ is given by
	\[
		\hat{\chi}(y_{i}t_{i}n)=\chi(n),
	\]
 for all $n\in N\cap H$ and $i \in \{ 1, \dots, \indexKpN\}$.\par
The first step is to obtain an expression of the conjugate of $\hat{\chi}$ by an element of $\Lp$. This is done in the 
following lemma.
\begin{lem}
\label{lem:conjugating_chi^}
Let $z \in \Lp$ and let $n \in N$ be such that $z =  y_i n$ for some $i \in \{ 1, \dots, \indexLpN\}$. Let moreover $n' \in N$ and 
$j \in \{ 1, \dots, \indexLpN\}$ 
be such that $y_j n' \in \leftexp{z}{H}$. Then
	\[
		\leftexp{z}{\hat{\chi}}(y_jn') = \chi ( t_{\kappa(i,j)}^{-1} \varphi_{\kappa(i,j)}^{-1}(n^{-1})d_{ij} \varphi_i^{-1}( n' ) n).
	\]
\end{lem}
	\begin{proof}
	We have $\leftexp{z^{-1}}(y_j n') =  n^{-1}y_i^{-1}y_j n'  y_i n$. Moreover
		\begin{align*}
			n^{-1}y_i^{-1}y_j n'  y_i n			& = n^{-1}y_i^{-1}y_j y_i \varphi_i^{-1}( n' ) n	\\
										& = n^{-1}y_{\kappa(i,j)} d_{ij} \varphi_i^{-1}( n' ) n\\
										& = y_{\kappa(i,j)} \varphi_{\kappa(i,j)}^{-1}(n^{-1})d_{ij} \varphi_i^{-1}( n' ) n\\
										& = y_{\kappa(i,j)}t_{\kappa(i,j)} t_{\kappa(i,j)}^{-1}\varphi_{\kappa(i,j)}^{-1}(n^{-1})d_{ij} \varphi_i^{-1}( n' ) n
		\end{align*}
	The element $y_{\kappa(i,j)} \varphi_{\kappa(i,j)}^{-1}(n^{-1})d_{ij} \varphi_i^{-1}( n' ) n$ is in $H$ by assumption so 
		\[
			t_{\kappa(i,j)}^{-1} \varphi_{\kappa(i,j)}^{-1}(n^{-1})d_{ij} \varphi_i^{-1}( n' ) n \in N\cap H.
		\] 
	Therefore $\hat{\chi}(n^{-1}y_i^{-1}y_j n'  y_i n) = \chi(t_{\kappa(i,j)}^{-1} \varphi_{\kappa(i,j)}^{-1}(n^{-1})d_{ij} \varphi_i^{-1}( n' ))$. 
	\end{proof}
Next we need to be able to express when an element of $\Kp$ belongs to a conjugate of $H$ in terms of conditions, which  
translate into a first order formula.
\begin{defn}
Let $i,j \in \{ 1, \dots, \indexKpN\}$ and $n, n' \in N$. We define $\mathbf{A}_{ij}(H, \chi, n, n')$ to be the predicate 
		\begin{multline*}
			\Big( \forall n'' \in \varphi_i( \leftexp{n}{(N\cap H)}) : {\varphi_{j}(\leftexp{n'}{n''})} \in \varphi_i (\leftexp{n}{(N\cap H)})\\ 
							\wedge\; \chi (\leftexp{n^{-1}}{\varphi_i}^{-1}(\varphi_{j}(\leftexp{n'}{n''}))) = \chi (\leftexp{n^{-1}}{\varphi_i}^{-1}(n''))\Big).
		\end{multline*}
\end{defn}
\begin{lem}
\label{lem:cond_H}
Let $i,j \in \{ 1, \dots, \indexKpN\}$ and $n, n' \in N$. Then $y_j n' \in \leftexp{y_i n}{H}$ if and only if 
$\mathbf{A}_{ij}(H, \chi, n, n')$ holds.
\end{lem}
	\begin{proof}
	Let $M$ be the normaliser of $N\cap \leftexp{y_i n}{H}$ in $\Kp$. We prove that 
        \[
            \leftexp{y_i n}{H} = \Stab_{M}(\leftexp{y_i n}{\chi}).
        \]
	This will be enough to conclude. Indeed $\mathbf{A}_{ij}$ is the conjunction of two predicates: the first 
	expresses exactly that $y_j n' \in M$ and the second means that $(y_j n')^{-1}$ fixes $\leftexp{y_i n}{\chi}$. Let therefore	
		\[
			A = \Stab_{M}(\leftexp{y_i n}{\chi}).
		\] 
	Clearly $N\cap \leftexp{y_i n}{H}$ is normal in $ \leftexp{y_i n}{H}$ and $\leftexp{y_i n}{H}$ fixes $ \leftexp{y_i n}{\chi}$, 
	so $ \leftexp{y_i n}{H} \subseteq A$. This inclusion gives that $\Kp =  \leftexp{y_i n}{H} N = A N $.
	Hence, by the second isomorphism theorem, we have that $\lvert \Kp : N \rvert =  \lvert A : N\cap A \rvert  
	= \lvert  \leftexp{y_i n}{H} : N\cap  \leftexp{y_i n}{H} \rvert$. Moreover, 
	by Mackey's formula, $N\cap  \leftexp{y_i n}{H} = N\cap A$ as $ \leftexp{y_i n}{\chi}$ induces irreducibly to $N$. 
	Thus 
	$A =  \leftexp{y_i n}{H}$ because $ \leftexp{y_i n}{H} \subseteq A$ and there is a subgroup that has the same index 
	in both.
	\end{proof}
We are now able to express $\cT_{\Lp,\Kp,\Gammap}(\widetilde{\theta}) = c'$ in terms 
of the pair $(H, \chi)$. Define
	\begin{align*}
		\onecocyp     &= \cocy{1}(L_p/N,\Func(K_p/N,W_{(p)}))\\	
		\onecobop     &=  \cobo{1}(L_p/N,\Func(K_p/N,W_{(p)}))	
	\end{align*}
By Lemma~\ref{lem:Z-BU_H-finite} every class $\coho{1}(L_{p}/N,F_{K_{p}})$ has a representative in 
$\onecocyp$. Moreover, let $\delta \in 
\onecocyp \cap \cobo{1}(L_p/N, F_{\Kp})$. Then there is an $\omega\in F_{K_p}$ such that for any $g\in L_p$, 
$\delta(gN) = \leftexp{g}{\omega}\omega^{-1}$.
Since $\leftexp{g}{\omega}\omega^{-1}$ has values in $W_{(p)}$ we have $\leftexp{g}{\omega}\omega^{-1} =
(\leftexp{g}{\omega}\omega^{-1})_{(p)}$, so by the properties of $f_{(\ell)}$ just before Lemma~\ref{lem:pi_ell_mu-cobound}, 
	\[
		\delta(gN) = \leftexp{g}{\omega_{(p)}}\omega^{-1}_{(p)},
	\]
hence $\delta \in \cobo{1}(L_p/N,\Func(K_p/N,W_{(p)}))$. Thus 
	\[
		\onecocyp \cap \cobo{1}(L_p/N, F_{\Kp})= \cobo{1}(L_p/N,\Func(K_p/N,W_{(p)}).
	\] 
It follows that the inclusion of $\onecocyp$
in $ \cocy{1}(L_p/N, F_{\Kp})$ induces an isomorphism 
	\[
		\coho{1}(L_{p}/N,F_{K_{p}})\cong \onecocyp/\onecobop.
	\]
We need the following definition for ease of notation.
\begin{defn}
Let $i,j \in \{1, \dots, \indexKpN\}$ and $k \in \{ 1, \dots, \indexLpN\}$. Let also  $\psi \in \Lin(G)$, $\nu \in \Gammap$, 
and $n, n' \in N$. We define $\mathbf{B}_{ijk}(H, \chi, \psi, \nu, n, n')$ to be the predicate
	\begin{multline*}
	 \chi (t_{\kappa(k,j)}^{-1} d_{kj}  \varphi_k^{-1} \big( n' \big) ) =\\ 
	 		\mu (y_k N)(y_{\kappa(i,j)} N)\, \cdot\, \delta(y_k N)(y_{\kappa(i,j)} N)\, \cdot \, \nu(y_{\kappa(i,j)} N)\,\cdot\, \psi(y_j) \psi(n')\\
				 \cdot \chi \big( t_{\kappa(i,j)}^{-1} \varphi_{\kappa(i,j)}^{-1}(n^{-1})d_{ij} \varphi_i^{-1}( n' ) n \big).
	\end{multline*}
\end{defn}
\begin{prop}
\label{prop:Linearisation_twist}
Let $(H,\chi) \in X_K$ be a pair corresponding to $\theta \in \Irr_{K}(N)$. Let $c' \in \coho{1}{(\Lp/N, F_{\Kp}/\Gammap)}$. Fix
$\mu\in \onecocyp$ such that, the $1$-cocycle in $\cocy{1}{(\Lp/N, F_{\Kp}/\Gammap)}$ 
defined by $g \mapsto \mu(g N)\Gammap$
is in the class $c'$. Then $\cT_{\Lp,\Kp,\Gammap}(\widetilde{\theta}) = c'$ if and only if there is a 
coboundary 	$\delta \in \onecobop$ such that, for all $k = 1, \dots, \indexLpN$, there are 
	\begin{enumerate}
	\renewcommand{\theenumi}{\em\alph{enumi})}
		\item $i \in \{ 1, \dots, \indexKpN \} $ and $n \in N$,
		\item 
			a character $\psi \in \Lin(G)$,
		\item a homomorphism $\nu \in \Gammap$,
	\end{enumerate}
such that, for all $ j \in \{ 1, \dots, \indexKpN\}$ and $n' \in N$ 
	\[
		\mathbf{A}_{ij}(H, \chi, n, n') \wedge \mathbf{A}_{kj}(H, \chi, 1, n') \Longrightarrow \mathbf{B}_{ijk}(H, \chi, \psi, \nu, n, n').
	\]
\end{prop}
	\begin{proof}
	We define $\widetilde{f}_{H}:\cocy{2}(H/(N\cap H))\rightarrow \cocy{2}(\Kp/N)$
	to be the isomorphism induced by pulling back cocycles along the isomorphism
	$\Kp/N\rightarrow H/(N\cap H)$. Let $\alpha$ be the factor set of $\hat{\chi}$. Let 
	$\hat{\alpha}$ be the (unique) cocycle in $\coho{2}(\Kp)$ descending to  
	$\widetilde{f}_{H}(\alpha)$. 
	The projective character
		\begin{equation}
		\label{eq:theta^_as_ind}
			\hat{\theta} = \Ind_{H, \hat{\alpha}}^{\Kp} \hat{\chi}
		\end{equation}
	is a strong extension of $\theta$. We therefore have $\cT_{\Lp,\Kp,\Gammap}(\widetilde{\theta}) = c'$ if and only if 
	there is a coboundary $\delta \in \cobo{1}{(\Lp/N, F_{\Kp})}$ such that, for all $\varphi_k \in C_{\Lp}$, there are 
	a degree one character $\psi \in \Lin(G)$ and a homomorphism $\nu \in \Gammap$ such that
		\[
			\leftexp{y_k}{\hat{\theta}} =  \hat{\theta} \cdot  \mu (y_k N)  \delta(y_k N) \nu \psi\lvert_{\Kp}.
		\]
	Substituting \eqref{eq:theta^_as_ind} in the last equation we obtain 
		\begin{equation}
		\label{eq:defining_mu_with_ind}
			\leftexp{y_k}{\big( \Ind_{H, \hat{\alpha}}^{\Kp} \hat{\chi} \big)} =  
					\big( \Ind_{H, \hat{\alpha}}^{\Kp} \hat{\chi} \big)
						 \cdot  \mu (y_k N)  \delta(y_k N) \nu \psi\lvert_{\Kp}.
		\end{equation}
	Let $\hat{\beta}$ be the factor set of $\leftexp{y_k}{\hat{\theta}}$. The left-hand side is equal to 
		\[
		 	\Ind_{\varphi_k(H), \hat{\beta}}^{\Kp} \leftexp{y_k}{\hat{\chi}}
		\]
	The right-hand side is equal to
		\[
			\Ind_{H, \hat{\beta}}^{\Kp} \big(  \hat{\chi}  \cdot  (\mu (y_k N)  \delta(y_k N) \nu \psi)\lvert_{H} \big).
		\]
	Note that $\mu (y_k N)  \delta(y_k N) \nu$ has factor set $\hat{\beta}\hat{\alpha}^{-1}$ because of \eqref{eq:defining_mu_with_ind}. 
	By Mackey's intertwining formula we therefore have that equation~\eqref{eq:defining_mu_with_ind} holds if and only if
	there are $n \in N$ and $i \in \{ 1, \dots, \indexKpN\}$ such that 
		\begin{multline*}
 				 \leftexp{y_k}{\hat{\chi}}\lvert_{(\varphi_k(H) \cap {\varphi_i(\leftexp{n}{H})})} 
 				        = \leftexp{y_i n}{\hat{\chi}}\lvert_{(\varphi_k(H) \cap{\varphi_i( \leftexp{n}{H)}})}  
					    \,   \leftexp{y_i n}{(\mu (y_k N)  \delta(y_k N) \nu \psi)}
					    \lvert_{(\varphi_k(H) \cap {\varphi_i(\leftexp{n}{H})})}.
		\end{multline*}
	In other words, since $\nu$ and $\psi$ are fixed by $\Kp$, equation~\eqref{eq:defining_mu_with_ind} holds if and only if
		\begin{multline}
			\label{eq:pred_before_B}
			\leftexp{y_k}{\hat{\chi}}(y_j n') = \\
				 \leftexp{y_i n}{\hat{\chi}}(y_j n')
					\cdot  \mu (y_k N)(y_{\kappa(i,j)} N) \cdot  \delta(y_k N) (y_{\kappa(i,j)} N) \cdot \nu(y_j N) \cdot  \psi(y_j) \psi(n'),
		\end{multline}
	for all $n' \in N$ and $j \in \{ 1, \dots, \indexKpN\}$ such that $y_j n' \in \varphi_k(H) \cap {\varphi_i(\leftexp{n}{H})}$. 
	By Lemma~\ref{lem:conjugating_chi^} we have 
		\begin{align*}
			\leftexp{y_k}{\hat{\chi}}(n' y_j)	 	&=  \chi (t_{\kappa(k,j)}^{-1}  d_{kj}  \varphi_k^{-1} \big( n' \big) )\\
			\leftexp{y_i n}{\hat{\chi}}(n' y_j)		&= \chi \big( t_{\kappa(i,j)}^{-1} \varphi_{\kappa(i,j)}^{-1}(n^{-1})d_{ij} 
										\varphi_i^{-1}( n' ) n \big).
		\end{align*}
	Substituting in \eqref{eq:pred_before_B} gives $\mathbf{B}_{ijk}(\chi, \psi, \nu, n, n')$. 
	Moreover, by Lemma~\ref{lem:cond_H}, $y_j n'  \in \varphi_k(H)$ if and only if 
	${\varphi_{j}(\leftexp{n'}{n})}  \in \varphi_k(N\cap H)$ and 
		\[
			\chi (\varphi_k^{-1}(\varphi_{j}(\leftexp{n'}{n''}))) = \chi (\varphi_k^{-1}(n''))
		\]
	for all $n'' \in \varphi_k(N\cap H)$. These two conditions form $\mathbf{A}_{kj}(H, \chi, n, n')$. Similarly,
	$y_j n'  \in  \leftexp{n}{\varphi_i(H)}$ if and only if $\mathbf{A}_{ij}(H, \chi, n, n')$ holds.\par 
	We finish the proof observing that $\chi, \psi, \mu(y_k N)$ and  $\nu$ all have values in $W_{(p)}$. 
	Thus if there is $\delta \in \cobo{1}{(\Lp/N, F_{\Kp})}$ satisfying the conditions above, then necessarily 
	$\delta \in \onecobop$. We may therefore restrict to $\delta \in \onecobop$ in the 
	equivalence statement.
	\end{proof}
\subsection{Definable sets for degree one characters of subgroups of $G$.}
For the rest of this paper $\Lan$ will denote the analytic Denef-Pas language and $\man$ will be the $\Lan$-structure 
$\man = (\Q_p, \Z \cup \{\infty\}, \mathbb{F}_p, \dots)$ (see \cite[\defMan]{Rationality1} for precise definitions).
The aim of the last sections is showing that the predicates we derived thus far are expressible as first order formulas in $\Lan$.
We will often derive formulas for the language of uniform pro-$p$ groups $\lan_N$ described in \cite[Definition~1.13]{duSautoy-rationality} 
and then use the fact that the $\lan_N$-structure $(N,\dots)$ is definably interpretable in  
$\man$ (see \cite[\lemInterpret]{Rationality1}).\par
From now onwards $\uniform$ will be a uniform open normal subgroup of $G$. 
We fix a minimal set of topological generators $n_1,\dots, n_d$ for 
$N$. By \cite[Proposition~3.7]{DdSMS}, $N$ is in bijection with $\Z_p^d$ 
via the map $(\lambda_1,\dots, \lambda_d)\mapsto n_1^{\lambda_1}\cdots n_d^{\lambda_d}$. If $g\in N$ is such that 
$ g = n_1^{\lambda_1}\cdots n_d^{\lambda_d}$ for some $\lambda_1, \dots, \lambda_d \in \Z_p$ we say that
$(\lambda_1,\dots, \lambda_d)$ are its $\Z_p$-coordinates (with respect to $n_1, \dots, n_d$). We recall that, 
as proved in \cite{duSautoy-rationality}, every open subgroup of $N$ admits a special set of $d$ topological 
generators called a {\em good basis} (cf.\ \cite[\defGoodBasis]{Rationality1}). This is key to expressing 
$\cH(K_p)$ as a definable subset of $\Z_p^{d(d+r)}$. We recall the definition of {\em basis} for a subgroup in $\cH(K_p)$
from \cite[Definition~2.10]{duSautoy-rationality} and
\cite[p.~261]{duSautoy-Segal-in-Horizons} 
(note that we use left cosets instead of du~Sautoy's right coset convention).
\begin{defn}
\label{def:basis}
	Let $H\in\cH(K_p)$. A $(d+\indexPN)$-tuple $(h_{1},\dots,h_{d},t_{1},\dots,t_{\indexPN})$
	of elements in $N$ is called a \emph{basis} for $H$ if
	\begin{enumerate}
		\item $(h_{1},\dots,h_{d})$ is a good basis for $N\cap H$, and
		\item $(y_{1}t_{1},\dots,y_{\indexPN}t_{\indexKpN})$ is a (left) transversal for $N\cap H$ 
		in $H$.
	\end{enumerate}
\end{defn}
We shall now show how to interpret predicates that involve quantifying on $\Lin(G)$ and other groups 
of characters. For this (and also for later), we need the following definition (taken from \cite[\propDc]{Rationality1}). 
Note that here (and later) we have denoted by $\iota$, the isomorphism $ \Q_p / \Z_p \to W_{(p)}$ 
defined as  $a/p^m+\Z_p \mapsto e^{2\pi i a/p^m}$.
\begin{defn}
	\label{def:D_c}
	Let $c \in \coho{2}(K_p/ \uniform)$, we define $\cD^c$ as the set of pairs 
	$(\tuple{\lambda}, \tuple{\xi})$, $\tuple{\lambda}\in \M_{d\times (d + \indexPN)}(\Z_p)$, $\tuple{\xi}=(\xi_1,\dots,\xi_d)\in \Q_p^{d}$  such that:
		\begin{enumerate}
			\item \label{pro:X_definable_1} the columns of $\tuple{\lambda}$ are the $\Z_p$-coordinates with respect to 
			$n_1, \dots, n_d$ of a basis $(h_1,\dots, h_d, t_1,\dots, t_\indexPN)$ for some subgroup 
			$H \in \cH(K_p)$.
			\item \label{pro:X_definable_2} The function $\{h_1,\dots, h_d\} \rightarrow \Q_p/\Z_p$, $h_i\mapsto \xi_i + \Z_p$, 
			extends to a (necessarily unique) continuous $H$-invariant homomorphism 
				\[
					\chi: N \cap H\longrightarrow \Q_p/\Z_p.
				\]
			\item \label{pro:X_definable_3} $\Ind_{N\cap H}^N (\iota \circ \chi) \in 
			\Irr_\stgroup(\uniform)$,
			\item \label{pro:X_definable_4} $\cC(H, (\iota \circ \chi)) = c$.
		\end{enumerate}
	\end{defn}
\propDc\ in \cite{Rationality1} shows that $\cD^c$ is a definable subset of $\Q_p^{d(d+r+1)}$
in $\man$.
\subsubsection{Definable set for twisting characters.}
We show that characters $\tau\in\Lin(N)$, such that $\tau = \psi\lvert_N$ 
for some $\psi \in \Lin(G)$, may be definably parametrised in $\man$, in a way that keeps track of the values 
of $\psi(y_i)$ for $i =1, \dots, \indexKpN$. Notice that, since $\Kp$ is a pro-$p$ group, every $\psi\in \Lin(G)$ 
is such that $\psi(y_i) \in W_{(p)}$ for all $i =1, \dots, \indexKpN$.\par
For the following lemma, and for later use, we define a 
function 
	\begin{equation}
	\label{eq:gamma_ext}
	\begin{gathered}
		\gamma: \{1,\dots,\indexGN\}^2\rightarrow \{1,\dots,\indexGN\}	\text{ and }	a_{ij} \in N \text{ by}\\
		y_{i}y_{j} = y_{\gamma(i,j)}a_{ij}.
	\end{gathered}
	\end{equation}
\begin{lem}
\label{lem:ext_N_G}
Let $\tau \in \Lin(N)$ and let $\sigma_1, \dots, \sigma_\indexKpN \in W_{(p)}$. 
Then $\tau = \psi\lvert_N$ for some $\psi \in \Lin(G)$ such that $\psi(y_i) = \sigma_i $ 
for $i =1, \dots, \indexKpN$, if and only if there are $\sigma_{\indexKpN + 1}, \dots, \sigma_{\indexGN} \in W_{(p)}$, 
such that, for $i, j \in \{1, \dots, \indexGN\}$ and all $n, n'\in N$,
		\[
			 \sigma_{\gamma(i,j)} \tau(a_{ij} \varphi_j^{-1}(n) n') = \sigma_i \sigma_j \tau(n) \tau (n').
		\]
\end{lem}
	\begin{proof}
	We have
		\begin{equation}
		\label{eq:multiplying_ys}
			y_i n y_j n'= y_i y_j y_j^{-1} n y_j n' = y_i y_j \varphi_j^{-1}( n ) n' = y_{\gamma(i,j)} a_{ij} \varphi_j^{-1}( n ) n' .
		\end{equation}
	Assume that $\psi \in \Lin(G)$ restricts to $\tau$. Then, since $\Kp$ is a pro-$p$ 
	group, $\psi$ and $\psi_{(p)}$ restrict to the same character of $K$ (hence also $\psi_{(p)}$ restricts to $\tau$). 
	Set $\sigma_i = \psi_{(p)}(y_i)$ (for $i = 1, \dots, \indexGN$). 
	On the one hand, $\psi(y_i n y_j n') = \sigma_i \sigma_j \tau(n) \tau (n') $
	and, on the other hand,
		\[
			\psi(y_i n y_j n') = \sigma_{\gamma(i,j)} \tau(a_{ij} \varphi_j^{-1}( n ) n' )
		\]
	 by \eqref{eq:multiplying_ys}.\par
	 %
	Conversely,  assume there exist  $\sigma_1, \dots, \sigma_{\indexGN} \in W_{(p)}$ such that 
	for $i, j \in \{1, \dots, \indexGN\}$, and all $n, n'\in N$, 
	$ \sigma_{\gamma(i,j)} \tau(a_{ij} \varphi_j^{-1}(n) n') = \sigma_i \sigma_j \tau(n) \tau (n')$. Then 
		\[
			\psi (y_i n) =  \sigma_i \tau(n) \qquad n \in N,\, i = 1, \dots, \indexGN
		\]
	defines a homomorphism $G \to W_{(p)}$. Indeed, 
		\[
			 \sigma_{\gamma(i,j)} \tau(a_{ij} \varphi_j^{-1}(n) n') = \sigma_i \sigma_j \tau(n) \tau (n') = \psi(y_i n)\psi(y_jn').
		\]
	Moreover, $\psi(y_i n y_j n') =  \sigma_{\gamma(i,j)} \tau(a_{ij} \varphi_j^{-1}(n) n')$ 
	by \ref{eq:multiplying_ys}. Thus we get $\psi(y_i n y_j n') =  \psi(y_i n)\psi_p(y_j n')$. Clearly $\psi\lvert_N = \tau$ and 
	we conclude. 
	\end{proof}
Let $\tuple{\lambda}_0 \in \M_{(d + r) \times d}(\Q_p)$ whose rows correspond to the $\Z_p$ coordinates of a
basis of $K_p$ (as an element of $\cH(K_p)$). Let $\cD^G$ be the projection on the $\tuple{\xi}$-component of the set 
	\[
		\{  (\tuple{\lambda}, \tuple{\xi}) \in \cD^1 \mid \tuple{\lambda} = \tuple{\lambda}_0 \}
	\]
where $\cD^1$ is as in Definition~\ref{def:D_c} for $K = G$. Clearly  $\cD^G$ is a definable subset of 
$\Q_p^d$ in $\man$. By definition, the first $d$ rows of $\tuple{\lambda}_0$ are the $\Z_p$-coordinates 
of a good basis of $N$. Thus $\cD^G$ is precisely the set of $d$-tuples $\tuple{\xi}$ 
such that the function 
	\[
		\{n_1,\dots, n_d\} \longrightarrow \Q_p/\Z_p, \qquad n_i\longmapsto \xi_i + \Z_p,
	\] 
extends to a (necessarily unique) continuous homomorphism 
$\tau : N \cap H\longrightarrow \Q_p/\Z_p$, such that $\iota \circ \tau = \psi\lvert_N$ for some 
$\psi \in \Lin(G)$.
We are now able to introduce the definable set parametrising twisting characters. 
For ease of notation we introduce the following definition.
\begin{defn}
An $\Lan$-formula whose free variables are exactly $x_1, \dots, x_n$ is called an {\em $\Lan$-condition} on $x_1, \dots, x_n$.
\end{defn}
\begin{prop}
\label{pro:X_definable-lin}
Let  $\cD^G_\Kp$ be the set of tuples of the form $(\xi_1, \dots, \xi_d, \sigma_1, \dots, \sigma_\indexKpN) 
\in \Q_p^{d + \indexKpN}$ such that the function 
	\[
		\{y_i n_j \mid i = 1, \dots, \indexKpN,\, j = 1, \dots, d \}\rightarrow \Q_p/\Z_p, \qquad y_i n_j \longmapsto \sigma_i + \xi_j + \Z_p,
	\] 
extends to a (necessarily unique) continuous homomorphism $\sigma: G \to \Q_p / \Z_p$. 
Then $\cD^G_\Kp$ is a 
definable set of $\Q_p^{d + \indexKpN}$ in $\man$. 
\end{prop}
	\begin{proof}
	By Lemma~\ref{lem:ext_N_G}, a tuple $(\xi_1, \dots, \xi_d, \sigma_1, \dots, \sigma_\indexKpN)$ is in $\cD^G_\Kp$ if and only 
	if 
		\begin{enumerate}
			\item \label{pro:X_definable-lin_pf_1}
			$\tuple{\xi} = (\xi_1, \dots, \xi_d) \in \cD^G$, 
			\item \label{pro:X_definable-lin_pf_2}
			and, denoting by $\tau: N \to \Q_p/\Z_p$ the homomorphism defined by $\tuple{\xi}$, there are 
			$\sigma_{\indexKpN + 1}, \dots, \sigma_\indexGN \in  \Q_p$ such that for 
			$i, j \in \{1, \dots, \indexGN\}$ and all $n, n'\in N$,
				\[
					 (\sigma_{\gamma(i,j)}   - \sigma_i -  \sigma_j) + \Z_p =
					 	 \tau(n) + \tau (n') -  \tau(a_{ij} \varphi_j^{-1}(n) n'),
				\]
			where $\gamma$ and $a_{ij}$ are as in \eqref{eq:gamma_ext}.
		\end{enumerate}
	Clearly, \ref{pro:X_definable-lin_pf_1} is a definable condition because $\cD^G$ is definable. 
	By the definable interpretation of $\struc_N$ in $\man$ and using $\tuple{\xi}$ to express the values of $\tau$, 
	we see that also \ref{pro:X_definable-lin_pf_2} is given by an $\Lan$-condition on 
	$ (\xi_1, \dots, \xi_d, \sigma_1, \dots, \sigma_\indexKpN)$.
	\end{proof}
\subsubsection{Definable set for $\Gamma_{\Kp, \tic{\theta}}$.}
By definition of $\cD^c$ there is a surjective map $\Psi: \cD^c \rightarrow \cC^{-1}(c)$ defined by 
$(\tuple{\lambda}, \tuple{\xi})\mapsto (H, \chi)$ where $H\in\cH(K_p)$ is the subgroup corresponding to 
the basis $(h_1,\dots, h_d, t_1,\dots, t_\indexPN)$ in \cite[\propDc\,(i)]{Rationality1}
and $\chi$ is as in \cite[\propDc\,(ii)]{Rationality1}.\par
Let $(\tuple{\lambda}, \tuple{\xi}) \in \cD^c$ and let 
$(H, \chi) = \Psi (\tuple{\lambda}, \tuple{\xi})$. Set $\theta = \Ind_{N \cap H}^N \chi$. We shall now 
produce a definable set that will be used to interpret predicates quantifying over $\Gamma_{\Kp, \tic{\theta}}$.
\begin{defn}
We define $\cD_{\Kp/N}$ as the set of tuples $(\sigma_1, \dots, \sigma_\indexKpN) \in \Q_p^{\indexKpN}$ 
giving a function $\Kp/N \rightarrow \Q_p/\Z_p$  defined by
	\[
		y_i N \longmapsto \nu_i + \Z_p \qquad \text{for } i \in \{ 1, \dots, \indexKpN\},
	\] 
extending to a homomorphism $\Kp/N \to \Q_p/\Z_p$.
\end{defn}
Clearly $\tuple{\nu} \in \cD_{\Kp/N}$  if and only if for $i, j \in \{1, \dots, \indexKpN\}$, $\sigma_{\gamma(i,j)} = 
\sigma_i + \sigma_j \mod \Z_p$. Thus $\cD_{\Kp/N}$ is a definable set. 
\begin{lem}
\label{lem:def_Gamma}
Let  $\cD_{\Kp}(\tuple{\lambda}, \tuple{\xi})$ be the set of tuples of the form 
$(\nu_1, \dots, \nu_\indexKpN) \in \Q_p^{\indexKpN}$ such that the function 
$\bar{\nu} : \Kp/N \rightarrow \Q_p/\Z_p$ defined by
	\[
		y_i N \longmapsto \nu_i + \Z_p\qquad \text{for } i \in \{ 1, \dots, \indexKpN\},
	\] 
is a homomorphism such that $\iota \circ \bar{\nu} \in \Gamma_{\Kp,\tic{\theta}}$. 
Then $\cD_{\Kp}(\tuple{\lambda}, \tuple{\xi})$ is a definable subset of $\Q_p^{\indexKpN}$ in $\man$. 
\end{lem}
	\begin{proof}
	We start by expressing the definition of $\Gamma_{\Kp, \tic{\theta}}$ in terms of $(H, \chi)$. 
	To do this, we need to fix a strong extension $\hat{\theta}$ of $\theta$ (all strong extensions are equally good as 
	the definition of $\Gamma_{\Kp, \tic{\theta}}$ does not depend on this choice). We choose the strong extension obtained 
	by inducing to $\Kp$ the projective character $\hat{\chi}$ of $H$ defined by
	\[
		\hat{\chi}(y_{i}t_{i}n)=\chi(n),
	\]
 	for all $n\in N\cap H$ and $i \in \{ 1, \dots, \indexKpN\}$. To say that 
	$\nu \in \Lin(\Kp/N)$ belongs to $\Gamma_{\Kp, \tic{\theta}}$ is to say that there is $\varepsilon \in \Lin(G)$ such that 
		\[
			\hat{\theta}\nu = \hat{\theta} \varepsilon\lvert_{\Kp}.
		\]
	By Mackey's formula, this happens if and only if there exist $\varepsilon \in \Lin(G)$, $i \in \{ 1, \dots, \indexKpN \}$  and $n \in N$ such that
		\[
		 	(\leftexp{y_i n}{\hat{\chi}}\,\nu)\lvert_{H \cap \leftexp{y_i n}{H}}
					= (\hat{\chi}\, \varepsilon)\lvert_{H \cap \leftexp{y_i n}{H}}.
		\]
	In other words, if and only if there exist $\varepsilon \in \Lin(G)$, $i \in \{ 1, \dots, \indexKpN \}$  and $n \in N$ such that for all $j \in \{ 1,\dots, \indexKpN\}$ 
	and all $n' \in N$ we have 
		\begin{equation}
		\label{eq:chi_has_Gamma}
			y_j n' \in H \cap \leftexp{y_i n}{H} \Longrightarrow  \leftexp{y_i n}{\hat{\chi}} (y_j n') \nu(y_j) = 
							\hat{\chi}(y_j n') \varepsilon(y_j) \varepsilon(n).
		\end{equation}
	We now rewrite \eqref{eq:chi_has_Gamma} in a way that involves only quantifying over $N$ and $\Lin(G)$, conjugation by the chosen 
	coset representatives of $N$ in $\Kp$, and values of $\chi$ on $N\cap H$. First we observe that, 
	by Lemma~\ref{lem:cond_H} we may replace the antecedent with the predicate
		\[
			\mathbf{A}_{1j}(H, \chi, 1, n') \wedge \mathbf{A}_{ij}(H, \chi, n, n').
		\]
	Secondly, by Lemma~\ref{lem:conjugating_chi^}, we may replace the consequent in 
	\eqref{eq:chi_has_Gamma} by the predicate $\mathbf{C}_{ij}(H, \chi, \nu, \varepsilon, n, n')$ defined as
		\[
			 \chi \big( t_{\kappa(i,j)}^{-1} \varphi_{\kappa(i,j)}^{-1}(n^{-1})d_{ij} \varphi_i^{-1}( n' ) n \big) \nu(y_j) = \chi(t_j^{-1} n') \varepsilon(y_j) 
			 	\varepsilon(n)
		\]
	We obtain that $\nu \in \Gamma_{\Kp, \tic{\theta}}$ if and only if the following predicate is true 
		\begin{multline}
		\label{eq:predicate_Gamma}
			\exists \varepsilon \in \Lin(G) : 
				\bigvee_{i \in \{1, \dots, \indexKpN\} } \Big( \exists n \in N\\
								 \bigwedge_{j \in \{ 1, \dots, \indexKpN\}}
								 	 \big( \forall n'\in N : \mathbf{A}_{kj}(H, \chi, 1, n') \wedge \mathbf{A}_{ij}(H, \chi, n, n')  
										\Longrightarrow \mathbf{C}_{ij}(H, \chi, n, n', \nu, \varepsilon) \big)  \Big).
		\end{multline}
	The last predicate may be written as an $\Lan$-condition on $\tuple{\nu}$ and $(\tuple{\lambda} , \tuple{\xi})$:
	\begin{itemize}
		\item[-] we use the interpretation of $\struc_N$ in $\man$ to express elements in $N$. 
		\item[-] We use tuples in $\cD_{\Kp/N}$ to express the values of $\nu$. 
		\item[-] We interpret $\exists \varepsilon \in \Lin(G)$ as
			\[
				\exists (\xi_i, \dots, \xi_d, \sigma_1, \dots, \sigma_\indexKpN) \in \cD_{\Kp}^{G},
			\]
			using $(\xi_1, \dots, \xi_d)$ to express $\varepsilon$ on $N$ and $(\sigma_1, \dots, \sigma_\indexKpN)$ 
			to express 
				\[
					\varepsilon(y_1),\dots, \varepsilon(y_\indexKpN).
				\] 
		\item[-] We interpret multiplication in $W_{(p)}$ as addition in $\Q_p/\Z_p$ and equality in $W_{(p)}$ 
		as equality in $\Q_p/\Z_p$ via $\iota$.
	\end{itemize}
    Writing \eqref{eq:predicate_Gamma} with these rules gives an $\Lan$-condition on $\nu_1, \dots, \nu_\indexKpN$; thus, 
    $\cD_{\Kp}(\tuple{\lambda}, \tuple{\xi})$ is definable. 
	\end{proof}
\subsection{Definable sets for $\onecocyp$ and $\onecobop$} 
In this subsection we describe the definable sets used to interpret predicates 
quantifying over $\onecocyp$ and $\onecobop$. 
	\begin{lem}
	\label{lem:schur_def_sets-twist}
	Let $\Omega$ be the surjective map from the set of matrices $\M_{\indexLpN \times \indexKpN}(\Q_p)$ to 
	the set of functions $L_p/N \rightarrow F_{K_p}$,
	defined by 
		\[
			\Omega((z_{ij})) = \big[ y_iN \longmapsto \iota \circ f_i\big],\quad \text{for}\  i\in \lbrace 1, \dots, \indexLpN \rbrace.
		\]
	where for each $i$, $f_i: \Kp/N \to \Q_p / \Z_p$ is the function $y_j N \mapsto z_{ij} + \Z_p$, for $j\in \lbrace 1, \dots, \indexKpN \rbrace$.
	Define $\widetilde{\mathcal{Z}} = \Omega^{-1}(\onecocyp)$ and  $\widetilde{\mathcal{B}} = \Omega^{-1}(\onecobop)$. 
	Then $\widetilde{\mathcal{Z}}$ and 
	$\widetilde{\mathcal{B}}$ are definable in $\man$.

	\end{lem}
		\begin{proof}
		We prove that the set $\widetilde{\mathcal{Z}}$ is definable. Let $\mathbf{z} \in\M_{r\times \indexLpN}(\Q_p)$. Then  
		$\Omega(\mathbf{z}) \in \onecocyp$ if and only if the following definable predicate in $\man$ holds:
		 for all $i,j \in \lbrace 1, \dots, \indexLpN \rbrace$ and $k \in \lbrace 1, \dots, \indexKpN \rbrace$, 
					\[
						z_{\gamma(i,j)\, k} =  z_{jk} + z_{\kappa(i,j)\, k} \mod \Z_p.
					\]
		This is obtained by just pulling back the $1$-cocycle identity through $\Omega$. 
		More precisely, by definition, $\Omega(\tuple{z})$ is a $1$-coboundary if, for all 
		$i,j \in \lbrace 1, \dots, \indexLpN \rbrace$,
			\[
				\Omega(\tuple{z})(y_i y_jN) =	     \Omega(\tuple{z})(y_iN)  \leftexp{y_i}{\Omega(\tuple{z})(y_jN)};
			\]
		that is, if $f_{\gamma(i,j)} = f_j + f_{\kappa(i,j)}$. This in turn is equivalent to the condition that, 
		for all $k \in \{ 1, \dots, \indexKpN\}$, 
			\[
				f_{\gamma(i,j)}(y_k)  =  f_j(y_k) + f_{\kappa(i,j)}(y_k),
			\]
		or equivalently, $z_{\gamma(i,j)\, k} =  z_{jk} + z_{\kappa(i,j)\, k} \mod \Z_p$.\par
		We prove that $\widetilde{\mathcal{B}}$ is definable. We need to express the condition for being 
		a $1$-coboundary. To this end, we parametrise a function $K_p/N\rightarrow \Q_p/\Z_p$ 
		by the $\indexKpN$-tuple $(b_1,\dots,b_\indexKpN)\in \Q_p^\indexKpN$
		representing its values on $y_1N,\dots, y_{\indexKpN}N$. Writing the $1$-coboundary condition 
		in terms of $b_1,\dots,b_\indexKpN$ we obtain that $\Omega(\mathbf{z}) \in \onecobop$ 
		is a $1$-coboundary if and only if there is $(b_1, \dots, \allowbreak b_\indexKpN) \in \Q_p^{r}$ 
		such that, for all  $i\in \lbrace 1, \dots, \indexLpN \rbrace$ and $j \in \{ 1, \dots, \indexKpN\}$,
			\[
				z_{ij}    = b_{\kappa(i,j)} - b_j   \mod    \Z_p.
			\]
		This predicate in $\man$ is given by an $\Lan$ condition on the $z_{ij}$'s and the $b_j$'s; thus, the set is definable.
		\end{proof}
\subsection{Definability of the predicate $\cT_{L_p, K_p, \Gamma_p}(\tic{\theta}) = c'$}
We are now ready to give an interpretation of $\twirrparams$ in the structure $\man$. In this subsection we 
will construct a definable set $\cD^{c,c'}$ corresponding to $\twirrparams$ up to a definable equivalence 
relation (which we shall introduce in the next subsection). This 
correspondence will be explicit and we will have a definable function (also introduced in the next 
subsection) giving the degree of the corresponding character for every element in $\cD^{c,c'}$.\par
We start with the following lemma which can be proved using the fact that  
twisting by degree one characters and induction are compatible (see for instance
the proof of \cite[Lemma~8.6(b)]{hrumar2015definable}).
 \begin{lem}
\label{lem:conj_induced_tw}
Let $M$ be a finite index subgroup of $N$, $\chi \in \Lin(M)$ and $\psi \in \Lin(G)$. 
Then, for all $g \in G$,
	\[
		\leftexp{g}{\big(\mathrm{Ind}_M^N \chi\big)} \psi\lvert_{N} 
				= \mathrm{Ind}_{\leftexp{g}{M}}^N(\leftexp{g}{\chi} \psi\lvert_{\leftexp{g}{M}}).
	\]
Moreover if $M'$ is another finite index subgroup of $N$ and $\chi, \chi'$ are degree one characters 
of $M$ and $M'$ respectively, such that $\Ind_M^N \chi$ and $\Ind_{M'}^N \chi'$ are irreducible, then 
$(\Ind_M^N \chi )\psi\lvert_N = \Ind_{M'}^N \chi'$ if and only if there exists $g \in N$ such that 
$(\Res^{\leftexp{g}{M}}_{\leftexp{g}{M} \cap M'} \leftexp{g}{\chi})
\psi\lvert_{\leftexp{g}{M} \cap M'} = \Res^{M'}_{\leftexp{g}{M} \cap M'}  \chi'$.
\end{lem}
     \begin{prop}
	\label{pro:X_definable_tw}
	Let $\cD^{c,c'}$ be the set of pairs 
	$(\tuple{\lambda}, \tuple{\xi}) \in \cD^{c}$ with the property that, for $(H, \chi) = \Psi (\tuple{\lambda}, \tuple{\xi})$, 
	$\chi$ induces to a character $\theta$ of $N$ such that $\tic{\theta} \in \twirrparams$. Then $\cD^{c,c'}$ is a definable 
	subset of $\Q_p^{d(d+r+1)}$
	 in $\man$.
	\end{prop}
	\begin{proof}
	Since $(\tuple{\lambda}, \tuple{\xi}) \in \cD^{c}$, we have that $\tic{\theta} \in \twirrparams$ if and only if
		\begin{enumerate}
			\item \label{pro:X_definable_tw_1}
				$\Stab_G( \tic{\theta} ) = L$
			\item \label{pro:X_definable_tw_2}
				$\Gamma_{K, \tic{\theta}} = \Gamma$
			\item \label{pro:X_definable_tw_3}
			$\cT_{L_p, K_p, \Gamma_p}(\tic{\theta}) = c'$.
		\end{enumerate}
	Let $\indexLN = \lvert L : N \rvert$. Up to reordering $(y_1,\dots, y_\indexGN)$, we may assume that 
			\[
			(y_1,\dots,y_\indexKpN, y_{\indexKpN +1}, \dots, y_\indexLN)
		\]
	is a left transversal of $N$ in $L$. We let $\varphi_i: N \to N$ be the the conjugation by $y_i$ for 
	$i = \{1, \dots, \indexLN\}$ and we define
		\[
			C_L = \{ \varphi_i \mid i = 1, \dots, \indexLN \}.
		\]
	By Lemma~\ref{lem:conj_induced_tw}, 
	$\Stab_G(\tic{\theta}) = L $ if and only the following statement holds:
		\begin{multline}
		\label{eq:stab_is_L}
			\forall\, \varphi \in C_G\ : \\
					\Big(\exists \psi\in \Lin(G) \big( (\Ind_{N\cap H}^{N}\chi) \psi\lvert_N = 
						\Ind_{\varphi (N\cap H)}^{N}\chi \circ\varphi^{-1} \big )
							\Longleftrightarrow \varphi \in C_L \Big).
	\end{multline}
	 Fix $\varphi \in C_G$. Lemma~\ref{lem:conj_induced_tw} with 
	$M = N \cap H$, $M' = \varphi(N \cap H)$ and $\chi' = \chi \circ\varphi^{-1}$ implies that 
	$(\Ind_{N\cap H}^{N}\chi)  \psi\lvert_N=\Ind_{\varphi(N\cap H)}^{N}\chi \circ\varphi^{-1}$ if and only if
		\begin{multline*}
			\exists\, g \in N,\ \forall\, h\in N\cap H: 
					 \big(\leftexp{g}{h} \in \varphi(N \cap H)\Longrightarrow 
							  \chi(h) \psi(h)
					= \chi \circ \varphi^{-1}(\leftexp{g}{h})\big).
		\end{multline*}
	We interpret $\exists \psi \in \Lin(G)$ as
			\[
				\exists (\psi_1, \dots, \psi_d) \in \cD^{G},
			\]
	using $(\psi_1, \dots, \psi_d)$ to express $\psi$ on $N$. We interpret multiplication in $W_{(p)}$ 
	as addition in $\Q_p/\Z_p$ and equality in $W_{(p)}$ as equality in $\Q_p / \Z_p$. Substituting in \eqref{eq:stab_is_L} shows that 
	there is an $\Lan$-condition on $(\tuple{\lambda}, \tuple{\xi})$ expressing $\Stab_G (\tic{\theta}) = L$.\par
	Next we show how to express $\Gamma_{K, \tic{\theta}} = \Gamma$ with an $\Lan$-condition on $(\tuple{\lambda}, \tuple{\xi})$. 
	First of all we notice that, by Proposition~\ref{prop:red_Gamma}, 
	$\Gamma_{\Kp, \tic{\theta}} = \Gamma_p$ if and only if $\Gamma_{K, \tic{\theta}} = \Gamma$. Thus it suffices to show that 
	$\Gamma_{\Kp, \tic{\theta}} = \Gamma_p$ gives rise to an $\Lan$-condition on $(\tuple{\lambda}, \tuple{\xi})$. 
	This is done using the definable set $\cD_{\Kp}(\tuple{\lambda}, \tuple{\xi})$ 
	in Lemma~\ref{lem:def_Gamma}. Indeed, $\Gamma_{\Kp, \tic{\theta}} = \Gamma_p$ if and only if, for all $\nu \in \Lin(\Kp/N)$, 
		\begin{equation}
			\label{eq:Gamma_Kp_is_Gamma_p}
			\nu \in \Gamma_{\Kp, \tic{\theta}} \iff  \nu \in \Gamma_p
		\end{equation}
	Let $\mathcal{Q}_{\Kp/N}$ be a set of representatives of the equivalence classes $\bmod$ $\Z_p$ in $\cD_{\Kp/N}$. 
	The set $\mathcal{Q}_{\Kp/N}$ is finite and therefore it is definable in $\man$. Notice that for all $\nu \in \Lin(\Kp/N)$, the 
	set $\mathcal{Q}_{\Kp/N}$ contains a unique tuple $\tuple{\nu} \in \cD_{\Kp/N}$ such that $\iota^{-1} \circ \nu$ is the function
		\[
			\Kp/N \longrightarrow \Q_p/\Z_p, \qquad y_i N \longmapsto \nu_i + \Z_p.
		\] 
	Let $\tuple{\Gamma}_p$ be the subset of $\mathcal{Q}_{\Kp/N}$ consisting of the tuples that correspond to 
	the homomorphisms in $\Gamma_p$.
	This allows us to express \eqref{eq:Gamma_Kp_is_Gamma_p}
	as an $\Lan$-condition on $(\tuple{\lambda}, \tuple{\xi})$, namely
		\[
			\forall \tuple{\nu} \in \mathcal{Q}_{\Kp/N}: \tuple{\nu} \in \cD_{\Kp}(\tuple{\lambda}, \tuple{\xi}) \iff \tuple{\nu} \in \tuple{\Gamma}_p.
		\]
	We prove that \ref{pro:X_definable_tw_3} is given by an $\Lan$-condition (on $(\tuple{\lambda}, \tuple{\xi})$). Fix
	$\mu\in \onecocyp$	such that, the $1$-cocycle on $\Lp/N$ defined by $g \mapsto \mu(g N)\Gammap$
	is in the class $c'$. By Proposition~\ref{prop:Linearisation_twist}, 
	$\cT_{L_p, K_p, \Gamma_p}(\tic{\theta}) = c'$ if and only if
		\begin{multline*}
			\exists \delta\in \onecobop: 
				\bigwedge_{k \in \{ 1, \dots, \indexLpN\}}  \exists n \in N\, \exists \psi \in \Lin(G)\, \exists \nu \in \Gamma_p : \\
			\bigvee_{i \in \{ 1, \dots, \indexKpN\}}
				\Big( \mathbf{A}_{kj}(H, \chi, 1, n') \wedge \mathbf{A}_{ij}(H, \chi, n, n') \Longrightarrow
				\mathbf{B}_{ijk}(H, \chi, \psi, \nu, n, n')\Big).
		\end{multline*}
	Now it suffices to write the last predicate as $\Lan$-condition on $(\tuple{\lambda} , \tuple{\xi})$:
	\begin{itemize}
		\item[-] we use the interpretation of $\struc_N$ in $\man$ to express elements and group operations in $N$. 
		\item[-] We use $\tuple{\xi}$ to express the values of $\chi$, as explained in \cite[\propDc]{Rationality1}.
		\item[-] We interpret the predicate $\exists \delta\in \onecobop$
		as $\exists \tuple{\delta} \in \tic{\cB}$ and we use $\delta_{k\, {\kappa(i,j)}}$ to express the value 
		$\delta(y_k N)(y_{\kappa(i,j)} N)$. 
		\item[-] By \ref{pro:X_definable_tw_2} we replace $\exists \nu \in \Gamma_p$ with 
		$\exists \nu \in \Gamma_{\Kp, \tic{\theta}}$ and we interpret the latter as $\exists \tuple{\nu} \in 
		\cD_{\Kp}(\tuple{\lambda}, \tuple{\xi})$, using $(\nu_1, \dots, \nu_\indexKpN)$ to express the 
		values of $\nu$. 
		\item[-] We interpret $\exists \psi \in \Lin(G)$ as
			\[
				\exists (\tau_1, \dots, \tau_d, \sigma_1, \dots, \sigma_\indexKpN) \in \cD_{\Kp}^{G},
			\]
		using $(\tau_1, \dots, \tau_d)$ to express $\psi$ on $N$ and $(\sigma_1, \dots, \sigma_\indexKpN)$ 
		to express $\psi(y_1),\dots,\allowbreak \psi(y_\indexKpN)$.
		\item[-] We interpret multiplication and equality in $W_{(p)}$ via $\iota$. 
	\end{itemize}
	This concludes the proof because 
	the sets $\tic{\cB}$, $\cD_{\Kp}(\tuple{\lambda}, \tuple{\xi})$, and $\cD_{\Kp}^{G}$ are definable in $\man$.
	\end{proof}
Proposition~\ref{pro:X_definable_tw} shows that $\Psi: (\tuple{\lambda}, \tuple{\xi})\mapsto (H, \chi)$ is a surjection from $\cD^{c,c'}$ to the set of pairs $(H, \chi)\in X_K$ such that $\theta = \Ind_{N\cap H}^N \chi$ satisfies $\tic{\theta}\in \twirrparams$.

\subsection{Finishing the proof of Theorem~\ref{thm:Main-twist}}
We write the partial zeta series as a generating function enumerating the equivalence classes of a family 
of definable equivalence relations. We conclude rationality of the partial twist zeta series
by \cite[\thmRationalSeries]{Rationality1}. Theorem~\ref{thm:Main-twist} then follows from 
Proposition~\ref{prop:partial-Main-twist}.\par
We start by constructing a definable equivalence 
relation on $\cD^{c,c'}$ whose equivalence classes will be in bijection with 
$\twirrparams$. 
Let $(\tuple{\lambda}, \tuple{\xi}), (\tuple{\lambda}', \tuple{\xi}') \in \cD^{c,c'}$ and
let $(H, \chi) = \Psi (\tuple{\lambda}, \tuple{\xi})$ and $(H', \chi') = \Psi (\tuple{\lambda}', \tuple{\xi}')$. We define an 
equivalence relation $\widetilde{\eqrel}$ on $\cD^{c,c'}$ by 
	\[
		((\tuple{\lambda}, \tuple{\xi}), (\tuple{\lambda}', \tuple{\xi}')) \in \widetilde{\eqrel}  \Longleftrightarrow 
	     \exists\, \psi\in \Lin(G),\  \mathrm{Ind}_{N\cap H}^N \chi = \mathrm{Ind}_{N\cap H'}^N (\chi'\psi|_N).
	\]
	\begin{lem}
	The relation $\widetilde{\eqrel}$ is definable in $\man$.
	\end{lem}
		\begin{proof}
		Let $(H, \chi), (H', \chi')$ be as above. By Lemma~\ref{lem:conj_induced_tw}, we have that
		$\Ind_{N\cap H}^N \chi = \allowbreak \Ind_{N\cap H'}^N \chi'(\psi|_N)$ for some $\psi \in \Lin(G)$ if and only if 
			\[
			   \exists\, \psi \in \Lin(G),\    \exists\, g \in N,\ \forall\, h\in N \cap H\ \left(\leftexp{g}{h} \in N \cap H' \Longrightarrow \chi(h) = (\chi'\psi|_N)(\leftexp{g}{h})\right).
			\]
	Using Proposition~\ref{pro:X_definable-lin} to parametrise $\psi|_N$ for $\psi\in \Lin(G)$ by points 
	in $\cD_\Kp^G$ and writing the above formula in the $\Z_p$-coordinates of $N$ we obtain an $\Lan$-formula defining 
	$\widetilde{\eqrel}$. Note that, as done before, we interpret multiplication and equality in $W_{(p)}$ 
	via $\iota^{-1}$.
		\end{proof}
Composing $\Psi$  with the surjective map $X_K\rightarrow \Irr_K(N)$ of \cite[\corSurjCoho]{Rationality1} 
induces a bijection between the set of equivalence classes $\cD^{c,c'}/\widetilde{\eqrel}$ and  $\twirrparams$. We now use 
this bijection to produce a definable family of equivalence relations giving the partial zeta series. 
For $(\tuple{\lambda}, \tuple{\xi}) \in \cD^{c,c'}$,
write $(h_1(\tuple{\lambda}), \dots, h_d(\tuple{\lambda}))$ for the good basis associated 
with $\tuple{\lambda}$ by \cite[\propDc]{Rationality1}.\par 
Let $\N_0$ denote the set of non-negative integers. The function $f:\cD^{c,c'}\rightarrow \N_0$ 
given by
	\[
	 (\tuple{\lambda}, \tuple{\xi}) \longmapsto \sum_{i = 1}^{d} \big(\omega(h_i(\tuple{\lambda})) - 1\big)
	\]
is definable in $\man$, as $\omega$ is interpreted as a definable function in $\man$ by \cite[\lemInterpret]{Rationality1}. 
If $\Psi(\tuple{\lambda}, \tuple{\xi}) = (H, \chi)$, 
then $p^{f(\tuple{\lambda}, \tuple{\xi})}$ is the degree of $\Ind_{N\cap H}^{N} \chi$.\par
%
The graph of $f$ gives a definable family 
$D^{c,c'}$ of subsets of  $\Q_p^{d(d+r+1)}$ and the equivalence relation 
$E^{c,c'} \subseteq D^{c,c'} \times D^{c,c'}$ defined by 
	\[
		((x, n), (x',n')) \in E^{c,c'} \iff (x,x') \in \widetilde{\eqrel}
	\] 
is a definable family of equivalence relations on $D^{c,c'}$. Note that, 
$n = n'$ whenever $(x,x') \in \widetilde{\eqrel}$ and 
so it is not necessary to add this condition to the definition of $E^{c,c'}$.\par
%
Since, for all $n\in \N_0$, the fibre $f^{-1}(n)$ is a union of equivalence classes for $\widetilde{\eqrel}$, 
the set $D^{c,c'}_n/E^{c,c'}_n$ is in bijection with the subset of characters of degree $p^n$ in $\twirrparams$.
It follows that
	\[
		\tpartial{N; L, K, \Gamma}{c, c'} = \sum_{n \in \N_0} \#(D^{c,c'}_n/E^{c,c'}_n) p^{-ns}.
	\]
Applying \cite[\thmRationalSeries]{Rationality1} to the series above
we deduce that $\tpartial{N; L, K, \Gamma}{c, c'}$ is a rational function in $p^{-s}$. This concludes the proof.

\begin{acknowledgement*}
The second author was supported by Project G.0792.18N of the Research Foundation - Flanders (FWO), 
by the Hausdorff Research Institute for Mathematics (Universit\"at Bonn), by the University of Auckland, and by Imperial College London.

The work on this paper was supported by a Durham University Travel Grant and LMS Scheme 4 grant 41678.
\end{acknowledgement*}
\bibliographystyle{alex}
\bibliography{IEEEfull,alex}

\end{document}